\documentclass[11pt,fleqn]{amsart}
\usepackage[paper=a4paper]{geometry}

\usepackage{hyperref}
\usepackage[utf8]{inputenc}
\usepackage[english]{babel}
\usepackage{enumerate}
\usepackage[osf,noBBpl]{mathpazo}
\usepackage[alphabetic,initials]{amsrefs}
\usepackage{mathtools}
\usepackage{amsfonts,amssymb,amsmath}

\usepackage{subfiles}

\usepackage{graphicx}
\usepackage[poly,arrow,curve,matrix]{xy}
\usepackage{helvet}
\usepackage{stmaryrd}
\usepackage{tikz}
\usepackage{mathdots}

\usetikzlibrary{patterns}

\newtheorem{thm}{Theorem}[section]
\newtheorem{lem}[thm]{Lemma}
\newtheorem{prop}[thm]{Proposition}
\newtheorem{cor}[thm]{Corollary}

\theoremstyle{definition}
\newtheorem{defn}[thm]{Definition}

\theoremstyle{remark}
\newtheorem*{rmk}{Remark}
\theoremstyle{remark}
\newtheorem*{ex}{Example}

\newcommand\NN{\mathbb N}
\newcommand\CC{\mathbb C}
\newcommand\QQ{\mathbb Q}

\newcommand\ZZ{\mathbb Z}

\newcommand\ot{\otimes}
\renewcommand\to{\rightarrow}
\renewcommand\phi{\varphi}

\newcommand\interval[1]{\llbracket #1 \rrbracket}

\newcommand\VV{\mathbb V}
\newcommand\WW{\mathbb W}

\newcommand\FF{\mathcal F}

\renewcommand\SS{\mathbf S}
\newcommand\Schur{\mathbb S}
\DeclareMathOperator\Ind{Ind}
\DeclareMathOperator\Res{Res}
\DeclareMathOperator\Hom{Hom}

\DeclareMathOperator\ssCoind{\underline{Coind}}
\DeclareMathOperator\ssHom{\underline{Hom}}

\newcommand\Mod[1]{#1\operatorname{-Mod}}

\DeclareMathOperator\soc{soc}

\newcommand\TT{\mathbb T}
\newcommand\CAT[1]{\mathcal{O}_{\mathsf{LA}}^{#1}}

\newcommand\gl{\mathfrak{gl}}
\renewcommand\sl{\mathfrak{sl}}
\newcommand\lie[1]{{\mathfrak{#1}}}

\renewcommand\epsilon{\varepsilon}
\newcommand\fin{\mathsf{fin}}
\newcommand\dual[1]{#1^\#}

\newcommand\Part{\mathsf{Part}}

\newcommand\LAprojector{\pi_{\mathsf{LA}}}

\DeclareMathOperator\tr{tr}

\newcommand\order{\mathsf{inf}}
\renewcommand\fin{\mathsf{fin}}
\DeclareMathOperator\End{End}
\DeclareMathOperator\Ext{Ext}

\title{Highest weight categories of $\gl(\infty)$-modules}
\author{Pablo Zadunaisky}
\address{Instituto de Investigaciones Matem\'aticas Luis A. Santal\'o (IMAS)
} 
\email{pzadub@dm.uba.ar}
\thanks{This paper was written while the author was a posdoc at Jacobs 
university under the DFG grant 980/8-1. The author is a CONICET researcher and 
received funding from BID PICT 0099/2019.}

\begin{document}

\begin{abstract}
We study a category of modules over $\gl(\infty)$ analogous to category 
$\mathcal O$. We fix adequate Cartan, Borel and Levi-type subalgebras $\lie h, 
\lie b$ and $\lie l$ with $\lie l \cong \gl(\infty)^n$, and define $\CAT{\lie 
l}{\gl(\infty)}$ to be the category of $\lie h$-semisimple, $\lie n$-nilpotent
modules that satisfy a large annihilator condition as $\lie l$-modules. Our 
main result is that these are highest weight categories in the sense of Cline, 
Parshall and Scott. We compute the simple multiplicities of standard objects 
and the standard multiplicities in injective objects, and show that a form of 
BGG reciprocity holds in $\CAT{\lie l}{\gl(\infty)}$. We also give a 
decomposition of $\CAT{\lie l}{\gl(\infty)}$ into irreducible blocks.

\noindent \textbf{Keywords:} Lie algebras, representations, $\gl(\infty)$, 
Large annihilator condition.
\end{abstract}

\maketitle

\section{Introduction}
The algebra $\gl(\infty)$ is the most basic example of a complex locally-finite 
Lie algebra, i.e. a limit of finite-dimensional complex Lie algebras. Its 
representation theory has been a very active area of study for the last ten 
years. See for example \cites{PS11a,PS11b,DCPS16,HPS19,GP19,PS19}.
Beyond its intrinsic interest the representation theory of $\gl(\infty)$ 
has connections to two other areas of Lie theory: the stable representation 
theory of the family $\gl(d,\CC)$, and the classical representation theory of 
the Lie superalgebra $\gl(n|m)$. 

The relation between representations of $\gl(\infty)$ and $\gl(n|m)$ goes back 
to Brundan \cite{Brundan03}, who used a categorical action of the quantized 
enveloping algebra of $\gl(\infty)$ on category $\mathcal O_{\gl(n|m)}$ to 
compute the characters of atypical finite dimensional modules. This approach 
was expanded by Brundan, Losev and Webster in \cite{BLW17} and by Brundan and 
Stroppel in \cite{BS12a}. Also, Hoyt, Penkov and Serganova studied the 
$\gl(\infty)$-module structure of the integral block $\mathcal O^{\ZZ}_{\gl
(n|m)}$ in \cite{HPS19}. Recently Serganova has extended this to a 
categorification of Fock-modules of $\sl(\infty)$ through the Deligne 
categories $\mathcal V_t$ in \cite{Serganova21}.

Let us now turn to the connections of $\gl(\infty)$ with stable representation
theory. Informally this refers to the study of sequences of representations 
$V_d$ of $\gl(d,\CC)$, compatible in a suitable sense, as $d$ goes to 
infinity. For example, set $W_d = \CC^d$ and $V_d(p,q) = W_d^{\ot p}\ot 
(W_d^*)^{\ot q}$. For each $d \geq 1$ the decomposition of this module into its 
simple components is a consequence of Schur-Weyl duality, and we can ask 
whether this decomposition is in some way compatible with the obvious 
inclusion maps $V_d(p.q) \hookrightarrow V_{d+1}(p,q)$. One way to solve this 
problem is regard the limit $T^{p,q} = \varinjlim_d V_d(p,q)$ as a
$\gl(\infty)$-module and study its decomposition, as done by Penkov and Styrkas 
in \cite{PS11b}. A more categorical approach to this problem appeared 
independently in the articles by Sam and Snowden \cite{SS15} and by Dan-Cohen, 
Penkov and Serganova in \cite{DCPS16}. The former studies the category 
$\operatorname{Rep}(\mathsf{GL}
(\infty))$ of stable algebraic representation theory of the family of groups 
$\mathsf{GL}_d(\CC)$, and the latter is focused on the category 
$\TT_{\gl(\infty)^0}$ of modules that arise as subquotients of the 
modules $T^{p,q}$. As mentioned in the introduction to \cite{SS15}, these  
categories are equivalent. 

By \cite{DCPS16}, objects in the category $\TT_{\gl(\infty)}^0$ are the 
integrable finite length $\gl(\infty)$-modules that satisfy the so-called 
large annihilator condition (LAC from now on). This category is analogous in 
many ways to the category of integrable finite dimensional representations of
$\gl(n,\CC)$, and it forms the backbone of the representation theory of 
$\gl(\infty)$. Since this category is well understood by now, it is natural to 
look for analogues of category $\mathcal O$ for $\gl(\infty)$. Several 
alternatives have been proposed such as those by Nampaisarn 
\cite{Nampaisarn17}, Coulembier and Penkov \cite{CP19}, and Penkov and 
Serganova \cite{PS19}.

The fact that there is no one obvious analogue of $\mathcal O$ is due to two 
reasons. First, Cartan and Borel subalgebras of $\gl(\infty)$ behave in a much 
more complicated way than in the finite dimensional case, in particular 
different choices of these will produce non-equivalent categories of highest 
weight modules. The second problem is that the enveloping algebra of 
$\gl(\infty)$ is not left-noetheran, so its finitely generated modules do not 
form an abelian category. The first problem can be solved (or rather swept 
under the rug) by fixing adequate choices of Cartan and Borel subalgebras; this
is the road followed in this paper. The second however does not have an obvious 
solution. 

In \cite{PS19} Penkov and Serganova propose replacing this condition
with the LAC. They study the category $\mathcal{OLA}$ consisting of $\lie 
h$-semisimple and $\lie n$-nilpotent modules satisfying the LAC, and in 
particular show that it is a highest weight category. On the other hand,
the categorifications arising in \cites{HPS19,Serganova21} above only satisfy 
the LAC when seen as modules over a Levi-type subalgebra $\lie l \subset 
\gl(\infty)$. Given the importance of the work relating representations of 
$\gl(\infty)$ to Lie superalgebras, this motivates the study of analogues of 
category $\mathcal O$ where the condition of being finitely generated is 
replaced by this weaker version of the LAC. Thus in this paper we study the 
category $\CAT{\lie l}{\gl(\infty)}$ of $\lie h$-semisimple, $\lie n$-torsion
modules satisfying the LAC with respect to various Levi-type subalgebras $\lie 
l$.

As mentioned above our main result is that $\CAT{\lie l}{\gl(\infty)}$ is a 
highest weight category. Let us give some details on the result. We 
introduce the notion of eligible weights, which are precisely those that can
appear in a module satisfying the LAC. Admissible weights are endowed with an
interval-finite order that has maximal but no minimal elements. Simple objects 
of $\CAT{\lie l}{\gl(\infty)}$ are precisely the simple highest weight modules 
indexed by eligible weights. Standard objects are given by projections of dual 
Verma modules to $\CAT{\lie l}{\gl(\infty)}$ and have infinite length. Their 
simple multiplicities are given in terms of weight multiplicities of a (huge) 
representation of $\gl(\infty)$. Finally, injective envelopes have finite 
standard filtrations that satisfy a form of BGG reciprocity. We point out that 
$\CAT{\lie l}{\gl(\infty)}$ does not have enough projectives, and hence no 
costandard modules.


\vspace{12pt}

The article is structured as follows. Section \ref{s:generalities} contains 
some general notation. In section \ref{s:faces} we introduce several 
presentations of $\gl(\infty)$, each of which highlights some particular 
features and subalgebras. Section \ref{s:rep-theory} deals with various matters 
related to representation theory of Lie algebras, and some specifics regarding 
$\gl(\infty)$. In section \ref{s:cats-reps} we discuss some basic categories of 
representations of $\gl(\infty)$ in its various incarnations. We begin our 
study of $\CAT{}{}$ in section \ref{s:ola}, where we classify its simple 
objects, prove some general categorical properties, and show that simple 
multiplicities of a general object can be computed in terms of simple 
multiplicities for category $\overline{\mathcal O}_{\lie s}$. In section 
\ref{s:standard} we prove the existence of standard modules and compute their 
simple multiplicities. This last result depends on a long technical computation 
given in an appendix to the section. Finally, in section \ref{s:injectives} 
we wrap up the proof that $\CAT{}{}$ is a highest weight category with an 
analysis of injective modules. We also prove an analogue of BGG reciprocity and 
give a decomposition of $\CAT{}{}$ into irreducible blocks.

The usual zoo of Cartan, Borel, parabolic, and Levi subalgebras, along with 
their nilpotent ideals, is augmented in each case by their corresponding 
exhaustions and the subalgebras spanned by finite-root spaces. To help keep 
track of these wild variety, most of these subalgebras are introduced at once in 
subsection \ref{ss:subalgebras-vn}, along with a visual device to describe them.

\section*{Acknowledgements}
I thank Ivan Penkov for introducing me to the study of locally-finite Lie 
algebras and posing the questions that led to this work. Also, Vera Serganova 
answered several questions and provided valuable references. 


\section{Generalities}
\label{s:generalities}

\subsection{Notation}
For each $r \in \NN$ we set $\interval r = \{1, 2, \ldots, r\}$. Throughout we 
denote by $\ZZ^\times$ the set of nonzero integers endowed with the following, 
not quite usual order
\begin{align*}
1 \prec 2 \prec 3 \prec \cdots \prec -3 \prec -2 \prec -1.
\end{align*}
The symbol $\delta_{i,j}$ denotes the Kronecker delta.

We denote by $\Part$ the set of all partitions, and we identify each partition
with its Young diagram. Given $\lambda \in \Part$ we will the corresponding 
Schur functor by $\Schur_\lambda$. For any $\lambda, \mu, \nu$ we denote by
$c_{\lambda,\mu}^\nu$ the corresponding Littlewood-Richardson coefficient.

Vector spaces and unadorned tensor products are always taken over $\CC$ unless 
explicitly stated. Given a vector space $V$ we denote its algebraic dual by 
$V^*$. Given $n \in \ZZ_{>0}$ we will denote by either $V^n$ or $n V$ the direct
sum of $n$ copies of $V$. Given $d \in \ZZ_{>0}$ we will denote by $\SS^d(V)$ 
the $d$-th symmetric power of $V$, by $\SS^\bullet(V)$ its symmetric 
algebra, and by $\SS^{\leq d}(V)$ the direct sum of all $\SS^{d'}(V)$ with $d' 
\leq d$. We will often use that given a second vector space $W$ we have
\begin{align*}
\SS^\bullet(V \ot W) &\cong \bigoplus_{\lambda \in \Part} 
	\Schur_\lambda(V) \otimes \Schur_\lambda(W).
\end{align*}

\subsection{Locally finite Lie algebras}

A locally finite Lie algebra $\lie g$ is one where any finite set of elements
is contained in a finite dimensional subalgebra. If $\lie g$ is countable 
dimensional then this is equivalent to the existence of a chain of finite 
dimensional subalgebras $\lie g_1 \subset \lie g_2 \subset \cdots \subset \lie 
g = \bigcup_{r \geq 0} \lie g_r$. Any such chain is called an \emph{exhaustion}
of $\lie g$. We say $\lie g$ is locally root-reductive if each $\lie g_r$ can 
be taken to be reductive, and the inclusions send root spaces to root spaces 
for a fixed choice of nested Cartan subalgebras $\lie h_r \subset \lie g_r$ 
satisfying $\lie h_r = \lie h_{r+1} \cap \lie g_r$. In general we will say that 
$\lie g$ is locally reductive, nilpotent, etc. if there is has an exhaustion 
such that each $\lie g_r$ is of the corresponding type.

We denote by $\gl(\infty)$ the direct limit of the sequence of Lie algebras
\begin{align*}
\gl(1,\CC) \hookrightarrow \gl(2,\CC) \hookrightarrow 
	\cdots \gl(n,\CC) \hookrightarrow \cdots
\end{align*}
where each map is given by inclusion in the upper left corner. This restricts
to inclusions of $\sl(n,\CC)$ into $\sl(n+1,\CC)$, and we denote the limit of
these algebras by $\sl(\infty)$, which is clearly a subalgebra of 
$\gl(\infty)$. This is a locally finite Lie algebra.

Fix a locally finite Lie algebra $\lie g$ and an exhaustion $\lie g_r$. Suppose 
also that we have a sequence $(M_r, j_r)$ where every $M_r$ is a $\lie 
g_r$-module (which makes every $M_s$ with $s \geq r$ a $\lie g_r$-module), and 
$j_r: M_r \to M_{r+1}$ is an injective morphism of $\lie 
g_r$-modules. The limit vector space $M = \varinjlim M_r$ has the structure of 
a $\lie g_r$-module for each $r$, and these structures are compatible in the 
obvious sense so $M$ is a $\lie g$-module. We refer to $M_r$ as an exhaustion 
of $M$. Any $\lie g$-module $M$ has an exhaustion: if $X$ is a generating set 
of $M$, we can take as $M_r$ the $\lie g_r$-submodule generated by $X$. 

Concepts such as Cartan subalgebras, root systems, Borel and parabolic 
subalgebras relative to these systems, weight modules, etc. are defined for 
locally finite Lie subalgebras. In the context of this paper these objects will 
behave as in the finite dimensional case, but there are some subtle differences 
which we will point out when relevant. For details we refer the reader to the
monograph \cite{HP22} and the references therein.

We will write $\lie g = \lie b \niplus \lie a$ if $\lie b$ is a subalgebra of
$\lie g$ and $\lie a$ is an ideal that is also a vector space complement for 
$\lie b$. Suppose we have a locally finite Lie algebra $\lie g$ with fixed 
splitting Cartan subalgebra $\lie h$ (i.e. such that $\lie g$ is a semisimple 
$\lie h$-module through the adjoint action), and Borel (i.e. maximal locally 
solvable) subalgebra $\lie b$ of the form $\lie h \niplus \lie n$ for some 
subalgebra $\lie n$. Given any
weight $\lambda \in \lie h^*$, we can define Verma modules as usual by 
$\Ind_{\lie b}^{\lie g} \CC_\lambda$. We will denote the Verma module by 
$M_{\lie g}(\lambda)$ or simply $M(\lambda)$ when $\lie g$ is clear from the 
context. We will also denote by $L(\lambda) = L_{\lie g}(\lambda)$ the 
corresponding unique simple quotient of $M_{\lie g}(\lambda)$.

\section{The many faces of $\lie{gl}(\infty)$}
\label{s:faces}

In this section we will review some standard facts about the Lie algebra 
$\gl(\infty)$ and introduce several different ways in which this arises. While 
this amounts to choosing different exhaustions, we will take a different 
approach to these various ``avatars'' of $\gl(\infty)$, which brings to the fore
some non-obvious structures in this Lie algebra.

\subsection{The Lie algebra $\lie g(\VV)$} 
Let $I$ be any infinite denumerable set and let $\VV = \langle v_i \mid i \in 
I\rangle$ and $\VV_* = \langle v^i \mid i \in I\rangle$, which we endow with a 
perfect pairing
\begin{align*}
\tr: \VV_* \ot \VV &\to \CC\\
	v^i \ot v_j &\mapsto \delta_{i,j}.
\end{align*}
The vector space $\VV \ot \VV_*$ is a nonunital associative algebra with 
$v' \ot v \cdot w' \ot w = \tr(v \ot w') v' \ot w$. We denote the Lie algebra 
associated to this algebra by $\lie g(\VV, \VV_*, \tr)$, or simply by $\lie 
g(\VV)$. For $i,j \in I$ we set $E_{i,j} = v^i \ot v_j$.

Take for example $I = \NN$. Given $r \in \ZZ_{>0}$ write $\lie g_r = \langle 
E_{i,j} \mid i,j \in \interval r\rangle$. Then $\lie g_r$ is a subalgebra 
of $\lie g(\VV)$ isomorphic to $\gl(r,\CC)$. The inclusions $\lie g_r \subset 
\lie g_{r+1}$ correspond to the injective Lie algebra morphisms $\gl(r,\CC) 
\hookrightarrow \gl(r+1,\CC)$ given by embedding a $r\times r$ matrix into the 
upper left corner of a $r+1 \times r+1$ matrix with its last row and column 
filled with zeroes. Thus
\begin{align*}
\lie g(\VV) 
	&= \bigcup_{r \geq 1} \lie g_r \cong 
		\varinjlim \gl(r,\CC) \cong \gl(\infty).
\end{align*}
The choice of $I$ does not change the isomorphism type of this algebra, so in 
general every algebra of the form $\lie g(\VV, \VV_*, \tr)$ is isomorphic to 
$\gl(\infty)$.

Notice that $\VV$ and $\VV_*$ are $\lie g(\VV)$-modules with $E_{i,j} e_k
= \delta_{j,k}e_i$ and $E_{i,j} e^k = - \delta_{k,i} e^j$. A simple comparison
shows that $\VV = \lim V_r$, where $V_r$ is the natural representation of 
$\gl(r,\CC)$ and the maps $V_r \to V_{r+1}$ are uniquely determined up to 
isomoprhism by the fact that they are $\gl(r,\CC)$-linear. In a similar fashion,
$\VV_* \cong \varinjlim V_r^*$.

\subsection{Cartan subalgebra and root decomposition}
The definition and study of Cartan subalgebras of $\gl(\infty)$ is an 
interesting and subtle question, which was taken up in \cite{NP03}. General
Cartan subalgebras can behave quite differently from Cartan subalgebras of 
finite dimensional reductive Lie algebras, for example they may not produce
a root decomposition of $\gl(\infty)$. We will sidestep this problem by 
choosing one particularly well-behaved Cartan subalgebra as ``the'' Cartan 
subalgebra of $\lie g(\VV)$, namely the obvious maximal commutative subalgebra 
$\lie h(\VV) = \langle E_{i,i} \mid i \in I \rangle$, and freely borrow 
notions from the theory of finite root systems that work for this particular 
choice. 

The action of the Cartan subalgebra $\lie h(\VV)$ on $\lie g(\VV)$ is 
semisimple. Denoting by $\epsilon_i$ the unique functional of $\lie h(\VV)^*$ 
such that $\epsilon_i(E_{j,j}) = \delta_{i,j}$, the root system of $\lie g(\VV)$
 with respect to $\lie h(\VV)$ is 
\begin{align*}
\Phi &= \{\epsilon_i - \epsilon_j \mid i,j \in I\}
\end{align*}
with the component of weight $\epsilon_i - \epsilon_j$ spanned by $E_{i,j}$, 
and hence $1$-dimensional. Also $\lie g(\VV)_0 = \lie h(\VV)$, which is of 
course infinite dimensional.

\subsection{Positive roots, simple roots and Borel subalgebras}
If we fix a total order $\prec$ on $I$ (as we did implicitly when setting $I = 
\NN$) we can see an element of $\lie g(\VV)$ as an infinite matrix whose rows 
and columns are indexed by the set $I$. A total order also allows us to define 
a set of positive roots
\begin{align*}
\Phi^+_I &= \{\epsilon_i - \epsilon_j \mid i \prec j \}
\end{align*}
and a corresponding Borel subalgebra $\lie b_I = \lie h(\VV) \oplus 
\bigoplus_{\alpha \in \Phi^+_I} \lie g(\VV)_\alpha$. 

Again, Borel subalgebras of $\gl(\infty)$ behave quite differently than Borel
subalgebras of $\gl(r,\CC)$, and their behaviour is tied to the order type of 
$I$. One striking difference is the following: simple roots can be defined as 
usual, namely as those positive roots which can not be written as the sum of 
two positive roots, but it may happen that simple roots span a strict subspace 
of the root space. In the extreme case $I = \QQ$ there are \emph{no} simple 
roots. In this article we will only consider a few well-behaved examples, the 
reader interested in the general theory of Borel subalgebras of $\gl(\infty)$
can consult \cites{DP04,DanCohen08,DCPS07}.

Up to order isomoprhism, the only cases where simple roots span the full root
space are $I=\ZZ_{>0}, \ZZ_{<0}$ and $\ZZ$ with their usual orders. These are 
called the \emph{Dynkin} cases, and the corresponding Borel subalgebras are said
to be of Dynkin type. An example of non Dynkin type comes from the choice $I = 
\ZZ^\times$, where the simple roots are
\begin{align*}
	\{\epsilon_i - \epsilon_{i+1} \mid i \in \ZZ^\times\setminus\{-1\}\},
\end{align*}
but the root $\epsilon_1 - \epsilon_{-1}$ is not in their span. We get a basis 
of the root space by adding this extra root to the set of simple roots. 
A root is called \emph{finite} if it is in the span of simple roots, otherwise 
it is called \emph{infinite}. By definition we are in a Dynkin case if and only 
if every root is finite.

\subsection{Further subalgebras and visual representation}
\label{ss:visual-rep}
For the rest of this article we will assume that the bases of $\VV$ and $\VV_*$ 
are indexed by $\ZZ^\times$. Thus elements of $\lie g(\VV)$ can be seen as 
infinite matrices with rows and columns indexed by $\ZZ^\times$ and finitely 
many nonzero entries. The choice of $\ZZ^\times$ as index set means that it 
makes sense to speak of the $k$-to-last row or column of an infinite matrix $g 
\in \lie g(\VV)$ for any $k \in \ZZ_{>0}$. 

We denote by $\lie b(\VV)$ Borel subalgebra corresponding to $I = \ZZ^\times$, 
which is non-Dynkin since it has an infinite root in its support. We also 
denote by $\lie n(\VV)$ the commutator subalgebra $[\lie b(\VV),\lie b(\VV)]$, 
so $\lie b(\VV) = \lie h(\VV) \niplus \lie n(\VV)$. Notice that $\lie n(\VV)$ 
is not nilpotent but only locally nilpotent.

We now introduce some useful subalgebras of $\lie g(\VV)$. Fix $r \in \ZZ_{>0}$ 
and set
\begin{align*}
\lie g(\VV)_r &= \langle E_{i,j} \mid i,j \in \pm \interval r\rangle
\end{align*}
This is a finite dimensional subalgebra of $\lie g(\VV)$, isomorphic to 
$\gl(2r, \CC)$. Clearly $\lie g(\VV)_r \subset \lie g_{r+1}(\VV)$ and $\lie 
g(\VV) =\bigcup_{r \geq 1} \lie g(\VV)_r$, so this is an exhaustion of $\lie 
g(\VV)$. We also set 
\begin{align*}
\VV[r] &= \langle v_i \mid r+1 \preceq i \preceq -r-1 \rangle \subset \VV
&\VV_*[r] &= \langle v^i \mid r+1 \preceq i \preceq -r-1 \rangle \subset \VV_*.
\end{align*}
The map $\tr$ restricts to a non-degenerate pairing between these two subspaces 
and we write $\lie g(\VV)[r] = \lie g(\VV[r], \VV_*[r], \tr)$. This Lie algebra 
is isomorphic to $\lie g(\VV)$, and it is the centraliser of $\lie g_k(\VV)$ 
inside $\lie g(\VV)$.

We will meet many more subalgebras of $\lie g(\VV)$, so we now introduce
a visual aid to recall them. We will represent elements of $\lie 
g(\VV)$ as squares, and we will represent a subalgebra by shading in gray the 
region where nonzero entries can be found, while unshaded areas will always be
filled with zeroes. The following examples should help clarify this idea.
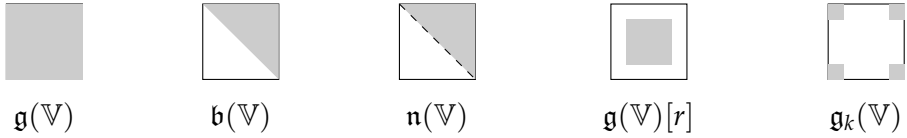
\begin{align*}
\begin{tikzpicture}
	\draw
		(0,0) rectangle (1,1); 
	\filldraw[color=black!20!white, draw=none]
		(0,0) rectangle (1,1);
	\draw (0.5,-0.5) node {$\lie g(\VV)$}; 
\end{tikzpicture}
&&\begin{tikzpicture}
	\draw
		(0,0) rectangle (1,1); 
	\filldraw[color=black!20!white, draw=none]
		(0,1) -- (1,1) -- (1,0) -- cycle;
	\draw (0.5,-0.5) node {$\lie b(\VV)$}; 
\end{tikzpicture}
&&\begin{tikzpicture}
	\draw
		(0,0) rectangle (1,1); 
	\filldraw[color=black!20!white, draw=none]
		(0,1) -- (1,1) -- (1,0) -- cycle;
	\draw[color=black,dashed] (0,1) -- (1,0);
	\draw (0.5,-0.5) node {$\lie n(\VV)$}; 
\end{tikzpicture}
&&\begin{tikzpicture}
	\draw
		(0,0) rectangle (1,1); 
	\filldraw[color=black!20!white, draw=none]
		(0.2,0.2) rectangle (0.8,0.8);
	\draw (0.5,-0.5) node {$\lie g(\VV)[r]$}; 
\end{tikzpicture}
&&\begin{tikzpicture}
	\draw
		(0,0) rectangle (1,1); 
	\filldraw[color=black!20!white, draw=none]
		(0,0) rectangle (0.2,0.2);
	\filldraw[color=black!20!white, draw=none]
		(0.8,0) rectangle (1,0.2);
	\filldraw[color=black!20!white, draw=none]
		(0,0.8) rectangle (0.2,1);
	\filldraw[color=black!20!white, draw=none]
		(0.8,0.8) rectangle (1,1);
	\draw (0.5,-0.5) node {$\lie g_k(\VV)$}; 
\end{tikzpicture}
\end{align*}
Notice that with these representations each corner of the square contains 
finitely many entries, and that the centre of the square concentrates 
infinitely many entries.

\subsection{The Lie algebra $\lie g(\VV^n)$}
\label{ss:vn}
Fix $n \in \NN$. We denote by $e_i$ the $i$-th vector of the canonical basis of 
$\CC^n$, and by $e^i$ the $i$-th vector in the dual canonical basis of 
$(\CC^n)^*$. Set $\VV^n = \VV \otimes \CC^n, \VV^n_* = \VV_* \ot (\CC^n)^*$
and set
\begin{align*}
	\tr^n: \VV^n_* \ot \VV^n &\to \CC\\
	(v^i \ot e^k)\ot(v_j \ot e_l) &\mapsto \delta_{i,j}\delta_{k,l}. 
\end{align*}
We can form the Lie algebra $\lie g(\VV^n) = \lie g(\VV^n, \VV^n_*, \tr^n)$ as 
before. Since $\VV$ and $\VV_*$ have bases indexed by $\ZZ^\times$, the vector
spaces $\VV^n$ and $\VV_*^n$ have bases indexed by $\ZZ^\times \times 
\interval n$. Since the isomorphism type of $\lie g(\VV)$ is independent of the 
indexing set as long as it is infinite and denumerable, any bijection between 
$\ZZ^\times$ and $\ZZ^\times \times \interval n$ will induce an isomorphism 
$\lie g(\VV^n) \cong \lie g(\VV)$. Notice that under any such 
isomorphism $\lie h(\VV^n)$ is mapped to $\lie h(\VV)$.

On the other hand there is an obvious isomorphism of vector spaces $\lie 
g(\VV^n) \cong \lie g(\VV) \ot M_n(\CC)$. Thus we can see elements of 
$\lie g(\VV^n)$ as $n \times n$ block matrices, with each block an infinite
matrix from $\lie g(\VV)$. We set for each $k \in
\interval n, r \in \ZZ_{>0}$
\begin{align*}
\VV^{(k)} &= \VV \otimes \langle e_k\rangle,
	&\VV^{(k)}_* &= \VV_* \otimes \langle e^k\rangle;\\
\VV^{(k)}[r] &= \VV[r] \otimes \langle e_k\rangle,
	&\VV^{(k)}_*[r] &= \VV_*[r] \otimes \langle e^k\rangle.
\end{align*}
so $\lie g(\VV^n) = \bigoplus_{k,l \in \interval n} \VV^{(k)} \ot \VV^{(l)}$.

We highlight this particular avatar of $\gl(\infty)$ since it reveals some 
internal structure which is not obvious in its usual presentation. The first
example is the Borel subalgebra $\lie b(\VV^n)$, which is awkward to handle as 
a subalgebra of $\gl(\infty)$ with the usual presentation. We also obtain
a non-obvious exhaustion by setting $\lie g_k(\VV^n) = \lie g_k(\VV) \ot 
M_n(\CC)$ for each $k \in \ZZ_{>0}$.

Another feature of $\gl(\infty)$ which becomes clear by looking at its 
avatar $\lie g(\VV^n)$ is that it has a $\ZZ^n$-grading compatible with the Lie 
algebra, inherited from the weight grading of $M_n(\CC)$ as 
$\gl(n,\CC)$-module. Thus for $k,l \in \interval n$ the direct summand 
$\VV^{(k)} \ot \VV^{(l)}_*$ is contained in the homogeneous component of degree 
$e_k - e_l$. This grading highlights some interesting subalgebras of block 
diagonal and block upper-triangular matrices, namely
\begin{align*}
\lie l 
	&= \bigoplus_{i = 1}^n \VV^{(k)} \ot \VV^{(l)}_*;&
\lie u
	& = \bigoplus_{i < j} \lie \VV^{(k)} \ot \VV_*^{(l)};&
\lie p
	&	= \bigoplus_{i \leq j} \lie \VV^{(k)} \ot \VV_*^{(l)},
\end{align*}
which are the subalgebras corresponding to $0$, strictly positive, and 
non-negative $\gl(n,\CC)$-weights, respectively. Notice the obvious Levi-type 
decomposition $\lie p = \lie l \niplus \lie u$, where $\lie l$ is not reductive 
but rather locally reductive, and $\lie u$ is $n-1$-step nilpotent.

\subsection{Subalgebras of $\lie g(\VV^n)$ and their visual representation}
\label{ss:subalgebras-vn}
As before, we will use some visual aids to describe the various subalgebras of
$\lie g(\VV^n)$. We will take $n = 3$ for these visual representations, and 
hence represent a subalgebra of $\lie g(\VV^n)$ by shading regions in a 
three-by-three grid of squares. We refer to these as the \emph{pictures} of the 
subalgebras. The following examples should clarify the idea.
\begin{align*}
\begin{tikzpicture}[scale=0.9]
	\filldraw[color=black!20!white, draw=none]
		(0,3) -- (3,3) -- (3,0) -- (0,3);
	\draw[color=black] (0,3) -- (3,0);
	\draw[step=1cm,black] (0.01,0.01) grid (2.99,2.99); 
	\draw (1.5, -0.5) node {$\lie b$};
\end{tikzpicture}
&& 
\begin{tikzpicture}[scale=0.9]
	\filldraw[color=black!20!white, draw=none]
		(0,3) -- (3,3) -- (3,0) -- (0,3);
	\draw[color=black,dashed] (0,3) -- (3,0);
	\draw[step=1cm,black] (0.01,0.01) grid (2.99,2.99); 
	\draw (1.5, -0.5) node {$\lie n$};
\end{tikzpicture}
\end{align*}

\begin{align*}
\begin{tikzpicture}[scale=0.9]
	\filldraw[color=black!20!white, draw=none]
		(0,3) -- (1,3) -- (1,2) -- (2,2) -- (2,1) -- (3,1) -- (3,0)
		-- (2,0) -- (2,1) -- (1,1) -- (1,2) -- (0,2) -- (0,3);
	\draw[step=1cm,black] (0.01,0.01) grid (2.99,2.99); 
	\draw (1.5, -0.5) node {$\lie l$};
\end{tikzpicture}
&&
\begin{tikzpicture}[scale=0.9]
	\filldraw[color=black!20!white, draw=none]
		(1,3) -- (3,3) -- (3,1) -- (2,1) -- (2,2) -- (1,2) -- (1,3);
	\draw[step=1cm,black] (0.01,0.01) grid (2.99,2.99); 
	\draw (1.5, -0.5) node {$\lie u$};
\end{tikzpicture}
&&
\begin{tikzpicture}[scale=0.9]
	\filldraw[color=black!20!white, draw=none]
		(0,3) -- (0,2) -- (1,2) -- (1,1) -- (2,1) -- (2,0) -- (3,0)
		-- (3,3) -- (0,3);
	\draw[step=1cm,black] (0.01,0.01) grid (2.99,2.99); 
	\draw (1.5, -0.5) node {$\lie p = \lie l \niplus \lie u$};
\end{tikzpicture}
\end{align*}
Recall the subalgebras $\lie g(\VV^n)_r$ for all $r \in \ZZ_{>0}$. We denote by 
$\lie g(\VV^n)[r]$ the subalgebra of $\lie g(\VV^n)_r$ invariants inside $\lie 
g(\VV^n)$, and by $\lie l[r]$ its intersection with $\lie l$.
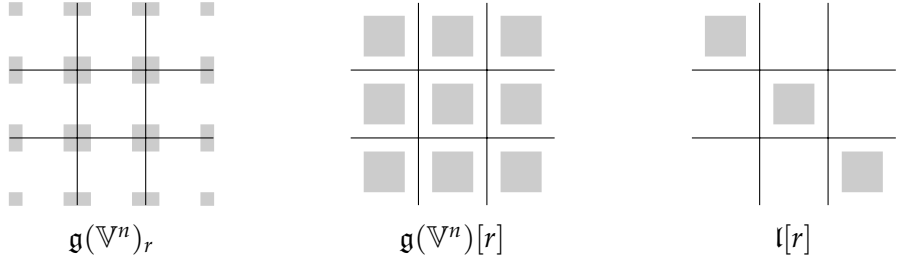
\begin{align*}
\begin{tikzpicture}[scale=0.9]
	\filldraw[color=black!20!white, draw=none]
		(0,3) rectangle (0.2,2.8) 
		(0.8,3) rectangle (1.2,2.8) 
		(1.8,3) rectangle (2.2,2.8)
		(2.8,3) rectangle (3,2.8)
		(0,2.2) rectangle (0.2,1.8) 
		(0.8,2.2) rectangle (1.2,1.8) 
		(1.8,2.2) rectangle (2.2,1.8)
		(2.8,2.2) rectangle (3,1.8)
		(0,1.2) rectangle (0.2,0.8) 
		(0.8,1.2) rectangle (1.2,0.8) 
		(1.8,1.2) rectangle (2.2,0.8)
		(2.8,1.2) rectangle (3,0.8)
		(0,0.2) rectangle (0.2,0) 
		(0.8,0.2) rectangle (1.2,0) 
		(1.8,0.2) rectangle (2.2,0)
		(2.8,0.2) rectangle (3,0);
	\draw[step=1cm,black] (0.01,0.01) grid (2.99,2.99); 
	\draw (1.5, -0.5) node {$\lie g(\VV^n)_r$};
\end{tikzpicture}
&&\begin{tikzpicture}[scale=0.9]
	\filldraw[color=black!20!white, draw=none]
		(0.2,2.8) rectangle (0.8,2.2)
		(1.2,2.8) rectangle (1.8,2.2)
		(2.2,2.8) rectangle (2.8,2.2)
		(0.2,1.8) rectangle (0.8,1.2)
		(1.2,1.8) rectangle (1.8,1.2)
		(2.2,1.8) rectangle (2.8,1.2)
		(0.2,0.8) rectangle (0.8,0.2)
		(1.2,0.8) rectangle (1.8,0.2)
		(2.2,0.8) rectangle (2.8,0.2);
	\draw[step=1cm,black] (0.01,0.01) grid (2.99,2.99); 
	\draw (1.5, -0.5) node {$\lie g(\VV^n)[r]$};
\end{tikzpicture}
&&\begin{tikzpicture}[scale=0.9]
	\filldraw[color=black!20!white, draw=none]
		(0.2,2.8) rectangle (0.8,2.2)
		(1.2,1.8) rectangle (1.8,1.2)
		(2.2,0.8) rectangle (2.8,0.2);
	\draw[step=1cm,black] (0.01,0.01) grid (2.99,2.99); 
	\draw (1.5, -0.5) node {$\lie l[r]$};
\end{tikzpicture}
\end{align*}
We also introduce the parabolic subalgebra $\lie p(r) = \lie l[r] + \lie b$. 
This algebra again has a Levi-type decomposition $\lie p(k) = \lie l[r]^+ 
\niplus \lie u(k)$, where $\lie l[r]^+ = \lie l[r] + \lie h$ and $\lie u(r)$ is 
the unique ideal giving this decomposition. 
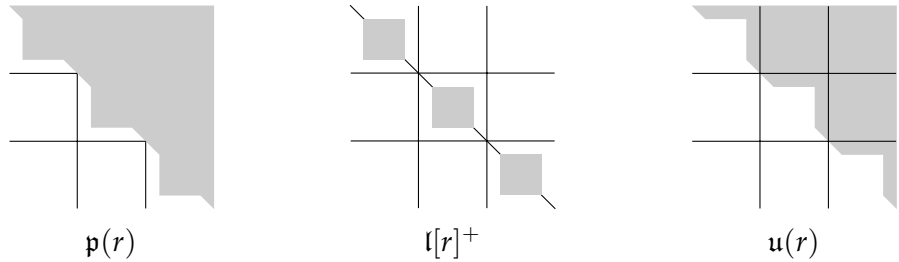
\begin{align*}
\begin{tikzpicture}[scale=0.9]
	\draw[step=1cm,black] (0.01,0.01) grid (2.99,2.99); 
	\filldraw[color=black!20!white, draw=none]
		(0,3) -- (0.2, 2.8) -- (0.2, 2.2) -- (0.8, 2.2) -- 
		(1.2,1.8) -- (1.2,1.2) -- (1.8,1.2) -- (2.2,0.8) -- (2.2,0.2)
		-- (2.8,0.2) -- (3,0) -- (3,3) -- (0,3);
	\draw (1.5, -0.5) node {$\lie p(r)$};
\end{tikzpicture}
&&
\begin{tikzpicture}[scale=0.9]
	\draw (3,0) -- (0,3);
	\filldraw[color=black!20!white, draw=none]
		(0.2,2.8) -- (0.2,2.2) -- (0.8,2.2) -- (0.8,2.8) -- cycle;
	\filldraw[color=black!20!white, draw=none]
		(1.2,1.8) -- (1.2,1.2) -- (1.8,1.2) -- (1.8,1.8) -- cycle;
	\filldraw[color=black!20!white, draw=none]
		(2.2,0.8) -- (2.2,0.2) -- (2.8,0.2) -- (2.8,0.8) -- cycle;
	\draw[step=1cm,black] (0.01,0.01) grid (2.99,2.99); 
	\draw (1.5, -0.5) node {$\lie l[r]^+$};
\end{tikzpicture}
&&
\begin{tikzpicture}[scale=0.9]
	\filldraw[draw=none,black!20!white]
		(0,3) -- (0.2,2.8) -- (0.8,2.8) -- (0.8,2.2) -- 
		(1,2) -- (1.2,1.8) -- (1.8,1.8) -- (1.8,1.2) -- 
		(2,1) -- (2.2,0.8) -- (2.8,0.8) -- (2.8,0.2) --
		(3,0) -- (3,3) -- (0,3);
	\draw[step=1cm,black] (0.01,0.01) grid (2.99,2.99); 
		\draw (1.5, -0.5) node {$\lie u(r)$};
\end{tikzpicture}
\end{align*}

Another family of algebras we will study are related to the finite and infinite
roots of $\lie g(\VV^n)$. We denote by $\lie s$ the subalgebra spanned by all
finite root spaces of $\lie g(\VV^n)$, and by $\lie q$ the parabolic subalgebra 
$\lie s + \lie b(\VV^n)$. This subalgebra has a Levi-type decomposition $\lie 
q = \lie s \niplus \lie m$, where $\lie m$ is the subspace spanned by all root 
spaces corresponding to positive infinite roots. 
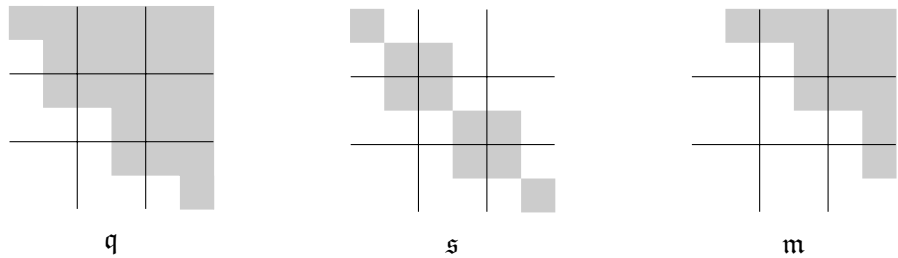
\begin{align*}
\begin{tikzpicture}[scale=0.9]
	\filldraw[color=black!20!white, draw=none]
		(0,3) rectangle (0.5,2.5) 
		(0.5,2.5) rectangle (1.5,1.5) 
		(1.5,1.5) rectangle (2.5,0.5)
		(2.5,0.5) rectangle (3,0)
		(0,3) -- (3,0) -- (3,3) -- cycle;
	\draw[step=1cm,black] (0.01,0.01) grid (2.99,2.99); 
	\draw (1.5, -0.5) node {$\lie q$};
\end{tikzpicture} 
&&
\begin{tikzpicture}[scale=0.9]
	\filldraw[color=black!20!white, draw=none]
		(0,3) rectangle (0.5,2.5) 
		(0.5,2.5) rectangle (1.5,1.5) 
		(1.5,1.5) rectangle (2.5,0.5)
		(2.5,0.5) rectangle (3,0);
	\draw[step=1cm,black] (0.01,0.01) grid (2.99,2.99); 
	\draw (1.5, -0.5) node {$\lie s$};
\end{tikzpicture} 
&&
\begin{tikzpicture}[scale=0.9]
	\filldraw[color=black!20!white, draw=none]
		(0.5,3) rectangle (3,2.5) 
		(1.5,2.5) rectangle (3,1.5) 
		(2.5,1.5) rectangle (3,0.5);
	\draw[step=1cm,black] (0.01,0.01) grid (2.99,2.99); 
	\draw (1.5, -0.5) node {$\lie m$};
\end{tikzpicture} 
\end{align*} 
We will also occasionally need the exhaustion of $\lie s$ given by $\lie s_r = 
\lie s \cap \lie g(\VV^n)_r$, whose picture is left for the interested reader.  

\subsection{Transpose automorphism}
Denote by $E_{i,j}^{(k,l)}$ the element $E_{i,j} \otimes e_k \ot e^l$.
The Lie algebra $\lie g(\VV^n)$ has an anti-automorphism $\tau$, analogous to
transposition in $\lie gl(r,\CC)$, given by $\tau(E_{i,j}^{(k,l)}) = 
E_{j,i}^{(l,k)}$. Given a subalgebra $\lie k \subset \lie g(\VV^n)$ its image 
by $\tau$ will also be denoted by $\overline{\lie k}$. Thus we have further 
subalgebras $\overline{\lie b}, \overline{\lie p}, \overline{\lie u}$, etc. 
Notice that the picture of the image of a subalgebra 
by $\tau$ is the reflection of the picture of the subalgebra by the main 
diagonal. We will denote by $T$ the automorphism of $U(\lie g(\VV^n))$ induced 
by $\tau$, and given $u \in U(\lie g(\VV^n))$ we write $\overline u$ 
for $T(u)$.

\subsection{Roots, weights and eligible weights}
We denote by $\epsilon^{(k)}_i$ the functional in $\lie h(\VV^n)^*$ given by
$\epsilon^{(k)}_i(E_{j,j}^{(l,l)}) = \delta_{i,j}\delta_{k,l}$. By a slight 
abuse of notation we can represent any element of $\lie h(\VV^n)^*$ as an 
infinite sum of the form $\sum_{i,k} a_i^{(k)}\epsilon_i^{(k)}$ with $a_i^{(k)}
\in \CC$. We denote by $\omega^{(k)}$ the functional $\omega^{(k)} = \sum_{i 
\in \ZZ^\times} \epsilon_i^{(k)}$. We set $\lie h(\VV^n)^\circ$ to be the span 
of $\{\epsilon^{(k)}_i, \omega^{(k)} \mid i \in \ZZ^\times, k \in \interval 
n\}$, and refer to its elements as eligible weights. A weight will be called 
$r$-eligible if it is a linear combination of the $\omega^{(k)}$ and the 
$\epsilon_i^{(k)}$ with $i \in \pm \interval r$.
As we will see in section \ref{s:rep-theory}, $r$-eligible weights 
parametrise the simple finite dimensional representations of $\lie l[r]^+$, 
which are all $1$-dimensional. The importance of these representations will be 
discussed in that same section.

As usual, $\Lambda$ denotes the $\ZZ$-span of the roots of $\lie g(\VV^n)$ 
inside $\lie h(\VV^n)^*$. It is a free $\ZZ$-module spanned by the families 
of roots
\begin{align*}
\epsilon_i^{(k)} &- \epsilon_{i+1}^{(k)} 
	&& k \in \interval n, i \in \ZZ^\times \setminus \{-1\}; \\
\epsilon_{-1}^{(k)} &- \epsilon_1^{(k+1)}
	&& k \in \interval{n-1}; \\
\epsilon_{-1}^{(k)} &- \epsilon_1^{(k)}
	&& k \in \interval n.
\end{align*}
The first two families consist of finite roots and any finite root is a linear 
combination of them, while the roots in the last family are infinite roots.

We denote by $\Lambda_\CC^+$ the $\CC$-span of the $\epsilon_i^{(k)}$. There is 
a pairing $(-, -) : \lie h(\VV^n)^* \otimes \Lambda_\CC^+ \to \CC$ with 
$(\lambda, \epsilon_i^{(k)}) = \lambda(E_{i,i}^{(k,k)})$. This allows us to 
define for any root $\alpha$ the reflection
\begin{align*}
	s_\alpha: \lie h(\VV^n)^* &\to \lie h(\VV^n)^*\\
	\lambda &\mapsto \lambda - 
		\frac{2 (\lambda,\alpha)}{(\alpha,\alpha)} \lambda.
\end{align*}
The group generated by these reflections is denoted by $\mathcal W$, and 
by analogy with the finite dimensional case is called the \emph{Weyl group} of
$\lie g(\VV^n)$.

Set $\lie h_r = \lie h(\VV^n) \cap \lie g(\VV^n)_r$ and for $\lambda \in \lie 
h(\VV^n)$ denote by $\lambda|_r$ its restriction to $\lie h_r$. Set also 
$\mathcal W_r$ to be the group generated by the reflections $s_\alpha$ 
with $\alpha$ a root of $\lie g(\VV^n)_r$. Clearly $\mathcal W$ is the 
union of the $\mathcal W_r$, and hence the direct limit of the Weyl groups of
the Lie algebras in the exhaustion $\{\lie g(\VV^n)_r\}_{r > 0}$. Furthermore,
if $\sigma \in \mathcal W_r$ and $\lambda \in \lie h(\VV^n)$ then 
$\sigma(\lambda)|_r = \sigma(\lambda|_r)$. The group $\mathcal W$ is 
isomorphic to the group of bijections of $\ZZ^\times \times \interval n$ that 
leave a cofinite set fixed, and its action on $\lie h(\VV^n)^*$ is given by the 
corresponding permutation of the $\epsilon_i^{(k)}$, suitably extended.

There is also an analogue of the dot action of the Weyl group for $\lie 
h(\VV^n)^*$. Set $\rho = \sum_{i,k} -i \epsilon_i^{(k)}$, and set for
every $\sigma \in \mathcal W$ and every $\lambda \in \lie h(\VV^n)^*$
\begin{align*}
	\sigma \cdot \lambda = \sigma(\lambda + \rho) - \rho.
\end{align*}
Both the usual and dot action send eligible weights to eligible weights but 
the dot action of the Weyl group of $\lie g(\VV^n)$ does not restrict to the 
dot action of $\lie g(\VV^n)_r$. However, if we denote by
$\mathcal W(\lie s)$ and $\mathcal W(\lie s_r)$ the groups generated by 
reflections associated to finite roots of $\lie g(\VV^n)$ and $\lie g(\VV^n)_r$ 
respectively, then $\mathcal W^\circ$ is the union of the $\mathcal 
W(\lie s_r)$ and the dot action of $\sigma \in \mathcal W(\lie s_r)$ on $\lie 
h(\VV^n)^*$ does satisfy that $(\sigma \cdot \lambda)|_r = \sigma \cdot
(\lambda|_r)$.

We define an order $<_\fin$ on weights defined as follows: $\mu <_\fin \lambda$
if and only if $\mu = \sigma \cdot \lambda$ for some $\sigma \in \mathcal 
W(\lie s)$ and $\lambda - \mu$ is a sum of positive roots, which necessarily 
will be finite. The following lemma will be used several times in the sequel. The argument originally appeared in the proof of \cite{PS19}*{Lemma 4.10}.
\begin{lem}
\label{lem:linked-weights}
Let $\lambda, \mu$ be $r$-eligible weights, with $\lambda \succeq \mu$. If
$\lambda|_s$ and $\mu|_s$ are linked for all $s \geq r$ then $\mu = \sigma 
\cdot \lambda$, with $\sigma \in \mathcal W(\lie s_{r+1})$. Thus $\lambda - 
\mu$ is a sum of positive finite roots of $\lie g(\VV^n)_r$, and in particular
$\mu <_{\fin} \lambda$.
\end{lem}
\begin{proof}
Denote by $\sigma_s$ the element in the Weyl group of $\lie g(\VV^n)_s$ such 
that $\sigma_s \cdot \lambda|_s = \mu|_s$, and by $\rho_s$ the half-sum of the 
positive roots of $\lie g_s$. If $\alpha$ is a simple root of $\lie 
g(\VV^n)_s$ that is also a root of $\lie l[r]$ then $(\mu, \alpha) = 0$, so 
$\sigma_s$ can only involve reflections $s_\beta$ with $\beta$ a root of $\lie 
g(\VV^n)_{r+1}$. It follows then that
\begin{align*}
0 = ( \lambda|_s - \sigma_s(\mu|_s), \alpha) 
	&= ( \sigma_s(\rho_s) - \rho_s, \alpha)
\end{align*}
for all $s$ and this is only possible if $\sigma_s$ only involves reflection 
through simple roots, i.e. if $\sigma_s \in \mathcal W(\lie s_{r+1})$.
\end{proof}

\section{Representation theory}
\label{s:rep-theory}
Throughout this section $\lie g$ is a Lie algebra with a fixed splitting Cartan 
subalgebra $\lie h$, and $\lie k$ is a root-subalgebra of $\lie g$ containing 
$\lie h$. 

\subsection{Socle filtrations}
Let $M$ be a $\lie g$-module. Recall that the socle of $M$ is the largest 
semisimple $\lie g$-submodule contained in $M$ and is denoted by $\soc M$. The 
socle filtration of $M$ is the filtration defined inductively as follows: 
$\soc^{(0)} M = \soc M$, and $\soc^{(n+1)} M$ is the preimage of 
$\soc(M / \soc^{(n)} M)$ by the quotient map $M \to M/ \soc^{(n)} M$. The 
\emph{layers} of the filtration are the  semisimple modules 
$\overline{\soc}^{(n)} M = \soc^{(n)} M / \soc^{(n-1)} M$.

\subsection{Categories of weight modules}
Given a $\lie g$-module $M$ and $\lambda \in \lie h^*$ we denote by $M_\lambda$ 
the subspace of $\lie h$-eigenvectors of eigenvalue $\lambda$, and refer to the 
set of all $\lambda$ with $M_\lambda \neq 0$ as the \emph{support} of $M$. 
We say $M$ is a \emph{weight} module if $M = \bigoplus_\lambda M_\lambda$. 
We denote by $\Mod{(\lie g, \lie h)}$ the full subcategory of $\Mod{\lie g}$ 
whose objects are weight modules.

The inclusion of $\Mod{(\lie g, \lie h)}$ in $\Mod{\lie g}$ is an exact functor 
with right adjoint $\Gamma_{\lie h}: \Mod{\lie g} \to \Mod{(\lie g, \lie h)}$ 
that assigns to each module $M$ the largest $\lie h$-semisimple submodule it 
contains. By standard homological algebra, this functor is left exact and sends 
injective objects to injective objects and direct limits to direct limits. In 
particular $\Mod{(\lie g, \lie h)}$ has enough injective objects. Given two 
weight modules $M, N$ we denote by $\Hom_{\lie g, \lie h}(M,N)$ the space of 
morphisms in the category of weight modules, and by $\Ext^\bullet_{\lie g, 
\lie h}(N,M)$ the corresponding derived functors.

Given any $\lie g$-modules $M, N$ the space $\Hom_{\lie g}(M,N)$ is a $\lie 
h$-module in a natural way, and we denote by $\ssHom_{\lie g}(M,N)$ the 
subspace spanned by its $\lie h$-semisimple vectors. If $M$ is a weight
representation then a map $\phi: M \to N$ is semisimple of weight $\lambda$ if
and only if $\phi(M_\mu) \subset N_{\lambda + \mu}$ for any weight $\mu$.

\subsection{Induction and coinduction for weight modules}
The restriction functor $\Res_{\lie k}^{\lie g}: \Mod{(\lie g, \lie h)} \to 
\Mod{(\lie k, \lie h)}$ has a left adjoint, given by the usual induction functor
$\Ind_{\lie k}^{\lie g}: \Mod{(\lie g, \lie h)} \to \Mod{(\lie k, \lie h)}$; 
both functors are exact. We often write just $N$ instead of $\Res_{\lie 
k}^{\lie g} N$. Restriction also has a right adjoint $\ssCoind_{\lie k}^{\lie 
g}: \Mod{(\lie g, \lie h)} \to \Mod{(\lie k, \lie h)}$, given by
\begin{align*}
\ssCoind_{\lie k}^{\lie g} M &= \ssHom_{\lie k}(U(\lie g), M)
\end{align*}
By definition a morphism $\phi: U(\lie g) \to M$ will be in $(\ssCoind_{\lie 
k}^{\lie g} M)_\mu$ if it is $\lie k$-linear and maps the weight space $U(\lie 
g)_\lambda$ to $M_{\lambda + \mu}$. In particular the semisimple coinduction
of $M$ only depends on $\Gamma_{\lie h}(M)$.

Unlike usual coinduction, semisimple coinduction is not exact. However it is 
left exact, sends injective objects to injective objects and direct limits to 
direct limits. This follows from the fact that it is right adjoint to an
exact functor.

\subsection{Semisimple duals}
Suppose we have fixed an antiautomorphism $\tau$ of $\lie g$ which restricts to 
the identity over $\lie h$ (in all our examples $\lie g$ will be $\gl(r,\CC)$ 
or $\gl(\infty)$, and $\tau$ will be the transposition map). Denote by $T$ the 
antiautomorphism of $U(\lie g)$ induced by $\tau$ and by $\overline{\lie k}$ 
the image of $\lie k$ by $\tau$. If $M$ is a $\lie k$-module with structure map 
$\rho: \lie k \to \End(M)$ we can define a $\overline{\lie k}$-module ${}^T M$ 
whose underlying vector space is $M$ and whose structure map is $\rho \circ \tau
:\overline{\lie k} \to \End(M)$. This assignation is functorial and commutes 
with the functor $\Gamma_{\lie h}$. The \emph{semisimple dual} of a $\lie 
k$-module $M$ is the $\overline{\lie k}$-module $M^\vee = {}^T \Gamma_{\lie h}
(M^*)$. The definition is analogous to the duality operator of category 
$\mathcal O$ for finite dimensional reductive Lie algebras. Indeed, we have 
$(M^\vee)_\lambda = (M_\lambda)^*$, and if $f \in M^\vee_\lambda$ then 
$(u \cdot f)(m) = f(T(u)m)$ for all $u \in U(\overline{\lie k})$ and $m \in M$.
We have the following result relating semisimple duals with induction and 
semisimple coinduction functors.

\begin{prop}
\label{prop:ss-dual-coind}
Let $\lie k \subset \lie g$ be a subalgebra containing $\lie h$, and let
$M$ be a weight $\lie k$-module. There is a natural isomorphism
\begin{align*}
(\Ind_{\lie k}^{\lie g} M)^\vee 
	&\cong \ssCoind_{\overline{\lie t}}^{\lie g} M^\vee. 
\end{align*}
In particular semisimple coinduction is exact over the image of the semisimple
dual functor.
\end{prop}
\begin{proof}
By \cite{Dixmier77}*{5.5.4 Proposition} there exists a natural isomorphism
$(\Ind_{\lie k}^{\lie g} M)^* \cong \ssCoind_{\lie t}^{\lie g} M^*$. If we apply
$\Gamma_{\lie h}$ and then twist by the automorphism $T$ we get 
\begin{align*}
(\Ind_{\lie k}^{\lie g} M)^\vee 
	&\cong \Gamma_{\lie h}\left({}^T \Hom_{\lie k}(U(\lie g), M^*) 
		\right).
\end{align*}
The automorphism $T$ defines an equivalence between the categories of weight 
$\lie k$ and $\overline{\lie k}$-modules, so there is a natural isomorphism 
$\Hom_{\lie k}(U(\lie g), M^*) \cong \Hom_{\overline{\lie k}}({}^T U(\lie g), 
{}^T M^*)$. Twisting the left $\lie g$-action on this module by $T$ corresponds
to twisting the right $\lie g$ action of $U(\lie g)$ by $T$, so
\begin{align*}
{}^T \Hom_{\lie k}\left( U(\lie g), M^* \right) &\cong 
			\Hom_{\overline{\lie k}}\left({}^T U(\lie g)^T, {}^T M^* \right)
		\cong \Hom_{\overline{\lie k}}(U(\lie g), {}^T M^*),
\end{align*}
where the last isomorphism comes from the fact that $T: U(\lie g) \to {}^T 
U(\lie g)^T$ defines an isomorphism of $U(\overline{\lie k})$-$U(\lie g)$
bimodules. Thus
\begin{align*}
{}^T \Gamma_{\lie h}\left(\Hom_{\lie k}(U(\lie g), M^*) \right)
	&\cong \ssHom_{\lie k}\left(
		U(\lie g), \Gamma_{\lie h}({}^T M^*))
	\right)
	= \ssCoind_{\overline{\lie k}}^{\lie g} M^\vee.
\end{align*}
The last statement follows from the fact that induction and semisimple duality
are exact functors.
\end{proof}

\subsection{Torsion and inflation}
Let $M$ be a $\lie g$-module. We say that $m \in M$ is $\lie 
k$-\emph{integrable} if for any $g \in \lie k$ the dimension of the span of 
$\{g^r m \mid r \geq 0\}$ is finite.
This happens in particular if $g^r m = 0$ for $r \gg 0$ depending on $g$, and
in this case we say that $M$ is $\lie k$-\emph{locally nilpotent}. We say that
$m$ is $\lie k$-\emph{torsion} if $g^r m = 0$ and $r$ can be chosen 
independently of $g$. We say that $M$ is $\lie k$-integrable, 
nilpotent, or torsion if each of its elements is.

The subspace of $\lie k$-torsion vectors of a $\lie g$-module $M$ is again a
$\lie g$-module, which we denote by $\Gamma_{\lie k}(M)$. This is a left exact
functor, as can be easily checked. Denoting by $J_r$ the right ideal generated
by $\lie k^r$ inside $U(\lie g)$, we see there exists a natural isomorphism
\begin{align*}
\Gamma_{\lie n} \cong \varinjlim \Hom_{\lie g}(U(\lie g)/J_r, -).
\end{align*}

Now suppose $\lie g = \lie k \niplus \lie r$. The projection $\lie g \to \lie 
g/\lie r \cong \lie k$ induces a functor
\begin{align*}
\mathcal I_{\lie k}^{\lie g}: \Mod{(\lie k, \lie h)} 
	\to \Mod{(\lie g, \lie h)}.
\end{align*}
We refer to this as the \emph{inflation} functor. It can be identified with 
$\ssHom_{\lie g}(U(\lie k), -)$ where $U(\lie k)$ is seen as a right $\lie 
g$-module, which implies it is exact and has a left adjoint $M \mapsto U(\lie 
k)\ot_{\lie g} M \cong M/U(\lie g) \lie r M$. We will write $M$ for $\mathcal 
I_{\lie k}^{\lie g}(M)$ when the context makes clear that we are seeing $M$ as
a $\lie g$-module.

\subsection{Weights, gradings and parabolic subalgebras}
Given any morphism of abelian groups $\phi: \lie h^* \to A$, we can turn an
semisimple $\lie h$-module into an $A$-graded vector space by setting $M_a =
\bigoplus_{\phi(\lambda) = a} M_\lambda$. We will denote this $A$-graded vector 
space by $M^\phi$, though we will often omit the superscript when it is clear
from context. The assignation $M^\phi$ is clearly functorial, and in particular
turns $\lie g$ into an $A$-graded Lie algebra, and any weight module into an
$A$-graded $\lie g$-module.

Suppose now that $A = \CC^n$ and that $\phi(\alpha) \in \ZZ^n$ for each root.
For $\chi, \xi \in \CC^n$ we write $\chi \succeq \xi$ if $\chi - \xi \in 
\ZZ_{\geq 0}^n$. We get a decomposition $\lie g = \lie g^\phi_{\prec 
0} \oplus \lie g^\phi_{0} \oplus \lie g^\phi_{\succ 0}$. If we have defined a
set of positive roots $\Phi^+$ along with a corresponding Borel subalgebra 
$\lie b$, and if $\phi(\alpha) \succ 0$ for every positive root, then the 
subalgebra $\lie g^\phi_{\succeq 0}$ is a parabolic subalgebra of $\lie g$.

\begin{ex}
\label{ex:gradings}
Fix $n \in \ZZ_{>0}$ and set $\lie g = \lie g(\VV^n)$.
\begin{enumerate}
\item We fix $n \in \NN$ and $\lie h ^\circ = \lie h(\VV^n)^\circ$. Consider the 
$\CC$-linear map $\phi: \lie h(\VV^n)^\circ \to \CC^n$ given by 
$\phi(\epsilon_i^{(k)}) = \phi(\omega^{(k)}) = e_k$, and extend to $\lie h^*$
arbitrarily. Since roots are mapped to vectors in $\ZZ^n$ this map induces a 
$\ZZ^n$-grading on $\gl(\infty)$, which coincides with the one introduced in 
\ref{ss:vn}. The corresponding parabolic subalgebra is $\lie g_{\succeq 0}^\phi 
= \lie p$ with $\lie g^\phi_0 = \lie l$ and $\lie g^\phi_{\succ 0} = \lie u$.

\item Fix $r > 0$ and denote by $\rho_r$ the sum of all positive roots of $\lie 
s_r$. Set $\theta: \lie h^\circ \to \ZZ$ to be the map $\alpha \mapsto
(\alpha, \rho_r) + (\phi(\alpha), \rho_n)$, where we are seeing $\rho_n$ as a
vector of $\CC^n$ in the usual way. The corresponding parabolic subalgebra is
$\lie p(r)$ with $\lie g^{\theta_r}_0 = \lie l[r]^+$ and $\lie g^{\theta_
r}_{>0} = \lie u(r)$.

\item Let $\psi: \lie h(\VV^n)^\circ \to \CC$ be the map $\psi(\epsilon^{(k)}_i)
= \frac12(n + 1 - 2k + \mathsf{sg}(i))$ (here $\mathsf{sg}(i)$ is $1$ if $i$ is 
a positive integer and $-1$ if it is negative) and $\psi(\omega^{(k)})= 0$. 
This formula has the nice property that if $\alpha$ is a finite root then 
$\psi(\alpha) = 0$, while the infinite roots $\epsilon^{(k)}_{-1} - 
\epsilon^{(k)}_1$ are mapped by $\psi$ to $1$. The corresponding parabolic 
subalgebra is $\lie q$ with $\lie g^\psi_0 = \lie s$ and $\lie g^\psi_{>0} = 
\lie m$.
\end{enumerate}
These maps and the induced gradings will reappear throughout the rest of the
article, so we fix the notations $\phi, \theta$ and $\psi$ for them.
\end{ex}

\subsection{Local composition series}
In general it is hard to establish a priori whether a $\lie g$-module
has a composition series. For this reason we introduce the following notion 
taken from \cite{DGK82}*{Proposition 3.2}. 
\begin{defn}
Let $M$ be a $\lie g$-module and let $\lambda \in \lie h^*$. We say that $M$ 
has a \emph{local composition series at $\lambda$} if there exist a finite 
filtration
\begin{align*}
M = F_0M \supset F_1M \supset \cdots \supset F_tM = \{0\}
\end{align*}
and a finite set $J \subset \{0, 1, \cdots, t-1\}$ such that
\begin{enumerate}[(i)]
\item if $j \in J$ then $F_{j} M / F_{j+1}M \cong L(\lambda_j)$ for some 
$\lambda_j \succeq \lambda$;
\item if $j \notin J$ and $\mu \succeq \lambda$ then $(F_{j-1} M / F_{j}M)_\mu 
= 0$.
\end{enumerate}
We say that $M$ has local composition series, or LCS for short, if it has a 
local composition series at all $\lambda \in \lie h^*$.
\end{defn}
It follows from the definition that a local composition series at $\lambda$ 
induces a local composition series at $\lambda'$ for any $\lambda' \succeq 
\lambda$, possibly with a different set $J$. A standard argument shows that if 
$M$ has two local composition series at $\lambda$ then the multiplicities of 
any simple object $L(\mu)$ in each series coincide. If $M$ has LCS then we 
denote by $[M:L(\mu)]$ this common multiplicity. The class of modules having 
LCS is closed under submodules, quotients and extensions, and multiplicity is 
additive for extensions.

\subsection{Finite dimensional representations of $\gl(\infty)$}
For each $a \in \CC$ the algebra $\gl(\infty)$ has a one-dimensional 
representation $\CC_a$, where $g \in \gl(\infty)$ acts by multiplication by 
$a\tr(g)$; in particular $\CC_a$ is a trivial $\sl(\infty)$-module. These are 
the only simple finite dimensional representations of $\gl(\infty)$, and any 
finite dimensional weight representation is a finite direct sum of these. 
In order to find more interesting representations, we need to look for infinite 
dimensional modules. Finitely generated modules of $\gl(\infty)$ do not form
an abelian category, since $U(\gl(\infty))$ is not left-noetherian, so we need 
to look for an alternative notion of a ``small'' $\gl(\infty)$-module. 

\subsection{The large annihilator condition}
Let $\lie k \subset \gl(\infty)$ be a subalgebra. Let $M$ be a
$\gl(\infty)$-module $M$ and let $m \in M$. We say that $m$ satisfies the 
\emph{large annihilator condition} (LAC from now on) with respect to $\lie k$ 
if there exists a finite dimensional subalgebra $\lie t \subset \lie k$ such 
that $\mathcal [\lie k^{\lie t},\lie k^{\lie t}]$, the derived subalgebra of 
the centraliser of $\lie t$ in $\lie k$, acts trivially on $\CC m$. In other 
words, the annihilator of $m$ contains the ``large'' subalgebra $[\lie 
k^{\lie t},\lie k^{\lie t}]$. We say that $M$ satisfies the LAC if every vector
in $M$ satisfies the LAC. 

The adjoint representation of $\gl(\infty)$ satisfies the LAC with respect to 
itself, and hence with respect to any other subalgebra $\lie k$. If $m \in M$ 
satisfies the LAC with respect to $\lie k$ then the $\gl(\infty)$ module 
generated by $m$ also satisfies the LAC. The tensor product of two 
representations satisfying the LAC again satisfies the LAC.

Denote by $\Mod{(\gl(\infty), \lie h)}_{\mathsf{LA}}^{\lie k}$ the full 
subcategory of modules satisfying the LAC with respect to $\lie k$. The natural 
inclusion of this category in $\Mod{(\gl(\infty), \lie h)}$ has a right 
adjoint, which we denote by $\Phi_{\lie k}$, or just $\Phi$ for simplicity. 
Being right adjoint to an exact functor, this functor is left exact and sends 
direct limits to direct limits and injective objects to injective objects. 

Let us consider the case $\lie k = \gl(\infty)$ in more detail. For simplicity 
we fix the exhaustion $\lie g(\VV)_r$ associated to the order $I = \ZZ^\times$. 
If $M$ satisfies the large annihilator condition and $m \in M$, the finite 
dimensional Lie algebra $\lie t$ is contained in $\lie g(\VV)_r$ for some $r 
\gg 0$, so we might as well take $\lie t = \lie g(\VV)_r$. Now the algebra 
$\gl(\infty)^{\lie t}$ is the subalgebra we denoted by $\lie g(\VV)[r]$ and is 
isomorphic to $\gl(\infty)$. Thus $\CC m$ is a $1$-dimensional representation 
of $\lie g(\VV)_r$ and hence isomorphic as such to $\CC_a$ for some $a \in 
\CC$.

\subsection{The large annihilator condition over $\lie l$}
Recall we have set $\lie l = \lie g_0$ for the $\ZZ^n$-grading introduced in 
subsection \ref{ss:vn}. In the sequel we will mostly consider modules 
satisfying the large annihilator condition with respect to $\lie l$, so we now 
focus on this condition. Since $\lie l = \lie g(\VV^{(1)}) \oplus \cdots \oplus 
\lie g(\VV^{(n)})$, its one-dimensional representations are external tensor 
products of the form $\CC_{a_1} \boxtimes \cdots \boxtimes \CC_{a_n}$, with 
$\chi = (a_1, \ldots, a_n) \in \CC^n$. 

A vector $m \in M$ satisfying the LAC with respect to $\lie l$ must span a 
one-dimensional module over $\lie l[r] \cong \lie l$, for some $r \gg 0$. This 
forces the weight of $m$ to be $r$-eligible and hence in $\lie h(\VV^n)^\circ$. 
In other words, the weight must be of the form
\begin{align*}
\sum_{k=1}^n a_k \omega^{(k)} +
	\sum_{k=1}^n\sum_{i\in \pm \interval r} b_i^{(k)} \epsilon_i^{(k)}.
\end{align*}
We call the $n$-tuple $\chi$ the \emph{level} of $m$, or rather of its weight.
Any vector in the module spanned by $m$ has the same level as $m$, and denoting 
by $M^\chi$ the space of all vectors of level $\chi$ we see that $M = 
\bigoplus_{\chi \in \CC^n} M^\chi$. Thus every vector in an indecomposable 
module $M$ satisfying the LAC with respect to $\lie l$ must have the same 
level, and we define this to be the level of $M$.

For each $r \geq 1$ we set $\Phi_r : \Mod{(\lie g, \lie h)} \to \Mod{(\lie g_r, 
\lie h_r)}$ to be the functor that assigns to a module $M$ the space of 
invariants $M^{\lie l[r]'}$. If $m \in M_\lambda$ spans a trivial $\lie 
l[r]'$-module then $\CC m \cong \CC_\lambda$ as $\lie l[r]^+$-modules. We can 
use this observation to get natural isomorphisms of functors
\begin{align*}
	\Phi_r &\cong \bigoplus_{\lambda} 
		\Hom_{\lie g} (U(\lie g) \ot_{\lie l[r]^+} \CC_\lambda, -); &
	\Phi &\cong \bigoplus_{\lambda \in \lie h^\circ} \varinjlim 
		\Hom_{\lie g} (U(\lie g) \ot_{\lie l[r]^+} \CC_\lambda, -),
\end{align*}
where the first sum is taken over the space of all $r$-eligible weights.

\section{Some categories of representations of $\gl(\infty)$}
\label{s:cats-reps}
Fix $n \in \NN$ and set $\lie g = \lie g(\VV^n)$. We omit the symbol $\VV^n$ 
from this point on and simply write $\lie b, \lie s, \lie g_r$, etc. 
for the subalgebras introduced in Subsection \ref{ss:subalgebras-vn}.

\subsection{Category $\overline{\mathcal O}$ for Dynkin Borel algebras}
\label{ss:nampaisarn}
In this subsection we review some results from Nampaisarn's thesis 
\cite{Nampaisarn17}, where he introduces and studies category 
$\overline{\mathcal O}$. We will state his results for the Dynkin 
subalgebra $\lie s \subset \lie g$ since this will be the only example we will 
need. We set $\lie b_{\lie s} = \lie b \cap \lie s$ and $\lie n_{\lie s} = 
[\lie b_{\lie s}, \lie b_{\lie s}]$. For each $r \in \ZZ_{>0}$ we set $\lie s_r
= \lie s \cap \lie g_r$, which is the Lie subalgebra of $\lie g_r$ spanned by 
finite roots, and is isomorphic to $\gl(r,\CC)^2 \oplus \gl(2r, \CC)^{n-1}$,
in particular it is reductive. 
\begin{defn}
An $\lie s$-module $M$ lies in $\overline{\mathcal O} = \overline{\mathcal 
O}^{\lie s}_{\lie b_{\lie s}}$ if it is $\lie h$-semisimple, $\lie n_{\lie 
s}$-torsion, and $\dim M_\lambda < \infty$ for each $\lambda \in \lie h^*$.
\end{defn}
The definition mimics that of $\mathcal O$ for finite dimensional reductive 
algebras. It makes sense for arbitrary locally reductive algebras with a 
fixed Borel subalgebra and many results from category $\mathcal O_{\lie s_r}$ 
extend to $\overline{\mathcal O}$ in a straightforward manner. On the other hand
the category lacks several obvious modules, for example the adjoint 
representation; if we waive the Dynkin hypothesis then Verma modules are also 
excluded. Still, it contains many interesting objects, and turns out 
to be an important stepping stone in the study of further categories of 
representations.

Since $\lie b_{\lie s}$ is a Dynkin Borel subalgebra, given $\lambda \in \lie 
h^*$ the Verma module $M_{\lie s}(\lambda)$ and its simple subquotient $L_{\lie 
s}(\lambda)$ belong to $\overline{\mathcal O}$. The category is also closed 
under semisimple duals, so the dual Verma module $M(\lambda)^\vee$ also belongs 
to $\overline{\mathcal O}$. A weight $\lambda \in \lie h^*$ is \emph{integral} 
if $(\lambda, \alpha) \in \ZZ$ for any simple root $\alpha$ of $\lie s$, 
\emph{dominant} if $(\lambda + \rho, \alpha) \notin \ZZ_{<0}$, and \emph{almost 
dominant} if $(\lambda + \rho, \alpha) \in \ZZ_{<0}$ for only finitely many
simple roots $\alpha$.

\begin{thm}[\cite{Nampaisarn17}*{Theorem 6.7, Proposition 8.11}]
\label{thm:dynkin-verma}
Let $\lambda, \mu \in \lie h^*$.
\begin{enumerate}[(a)]
\item If $M(\mu) \subset M(\lambda)$ then $\lambda \succeq \mu$ and there 
exists $\sigma \in \mathcal W(\lie s)$ such that $\mu = \sigma \cdot \lambda$.

\item If $M$ is a highest weight module with highest weight vector $m \in 
M_{\lambda}$ then for $r \gg 0$
\begin{align*}
[M: L_{\lie s}(\mu)] &= [U(\lie s_r) m: L_{\lie s_r}(\mu|_r)].
\end{align*}
\end{enumerate}
\end{thm}
It follows that $m(\lambda, \mu) = [M_{\lie s}(\lambda): L_{\lie s}(\mu)]$
coincides with the Kazhdan-Lusztig multiplicities of the corresponding Verma
modules for the finite-dimensional reductive algebra $\lie s_k$.

Let $\lambda$ be a dominant weight and denote by $\overline{\mathcal O}
[\lambda]$ the subcategory of the modules in $\overline{\mathcal O}$ whose 
support is contained in $\{\lambda - \mu \mid \mu \succeq 0\}$. It follows from 
the previous theorem that $\overline{\mathcal O}[\lambda]$ is a block of 
category $\overline{\mathcal O}$. By \cite{Nampaisarn17}*{Theorem 9.9} this 
block has enough injectives, and by \cite{Nampaisarn17}*{Proposition 9.21} 
these injective objects have finite filtrations by dual Verma modules (notice 
Nampaisarn refers to dual Verma modules as costandard modules), and their multiplicities are given by BGG-reciprocity. We put this down as a theorem for 
future reference.

\begin{thm}[\cite{Nampaisarn17}*{Theorem 9.9, Proposition 9.21}]
\label{thm:inj-filtration}
Let $\lambda$ be an almost dominant weight. Then $L_{\lie s}(\lambda)$ has an
injective envelope $I_{\lie s}(\lambda)$ in $\overline{\mathcal O}$. The 
injective envelope has a finite filtration whose layers are modules of the form 
$M(\lambda^{(k)})^\vee$ with $i = 0, \ldots, m$, and such that $\lambda^{(0)} =
\lambda$ and $\lambda^{(k)} >_\fin \lambda$. Furthermore, the multiplicity of
$M(\mu)^\vee$ in this filtration equals $m(\lambda, \mu)$.
\end{thm}

\subsection{The categories $\TT_{\lie g}$ and $\TT_{\lie l}$}
We now return to our study of representations of $\gl(\infty)$. Throughout this
section we identify $\gl(\infty)$ with $\lie g(\VV)$, and so $\VV$ and $\VV_*$
are $\gl(\infty)$-modules. 
A tensor module of $\gl(\infty)$ is a subquotient of a finite direct sum of 
modules of the form $\VV^p \ot \VV_*^q$ for $p,q \in \NN_0$. 
Tensor modules were first studied by Penkov and Styrkas in \cite{PS11b} and
later by Dan-Cohen, Penkov and Serganova in \cite{DCPS16}; in \cite{SS15} Sam and Snowden study the equivalent category $\operatorname{Rep}^{st} 
\mathbf{GL}(\infty)$. It follows from \cite{DCPS16}*{section 4} that a module 
is a tensor module if and only if it is an integrable module of finite length 
satisfying the LAC with respect to $\gl(\infty)$ and of level $0$. This 
category is denoted by $\TT_{\gl(\infty)}^0$.

The simple tensor modules are parametrised by pairs of partitions $(\lambda, 
\mu)$. The corresponding simple module, which we denote by $L(\lambda, \mu)$, 
is the simple highest weight module with respect to the Borel subalgebra $\lie 
b(\VV)$ with highest weight $\lambda_1 \epsilon_1 + \cdots +\lambda_n 
\epsilon_n - \mu_n \epsilon_{-n} - \cdots - \mu_1 \epsilon_{-1}$. The module 
$I(\lambda, \mu) = 
\Schur_\lambda(\VV) \ot \Schur_\mu(\VV_*)$ belongs to $\TT_{\gl(\infty)}^0$ and 
is injective in this category \cite{DCPS16}*{Corollary 4.6}. The layers of its
socle filtration are given by \cite{PS11b}*{Theorem 2.3}
\begin{align*}
	\overline{\soc}^{(r+1)} I(\lambda,\mu)
		&= \bigoplus_{\gamma \vdash r, \lambda',\mu'} 
			c^{\lambda}_{\lambda',\gamma}c^{\mu}_{\mu',\gamma}
				L(\lambda',\mu').
\end{align*}
In particular $I(\lambda, \mu)$ is the injective envelope of $L(\lambda, 
\mu)$ in $\TT_{\gl(\infty)}^0$.

\begin{lem}
\label{lem:tensor-fg-fl}
Let $M$ be an integrable weight $\gl(\infty)$-module satisfying the large 
annihilator condition with respect to $\gl(\infty)$. Then $M$ has finite length 
if and only if $M$ is finitely generated.
\end{lem}
\begin{proof}
We will use the isomorphism $\gl(\infty) \cong \lie g(\VV)$ to see $M$ as a 
$\lie g(\VV)$-module. If $M$ has finite length then it is finitely generated. 
To prove the reverse implication it is enough to prove it when $M$ is cyclic, 
so we assume that $M$ is generated by $v \in M_\lambda^a$ for $\lambda \in 
\lie h^*$ and $a \in \CC$. 

Take $r \gg 0$ so $v$ spans a $1$-dimensional $\lie g(\VV)[r]$-module 
isomorphic to $\CC_a$. Take $\lie r$ to be the locally nilpotent ideal of 
$\lie g(\VV)[r] + \lie b$, and $\overline{\lie r}$ to be the opposite ideal, so 
$\gl(\infty) = \lie r \oplus \lie g(\VV)[r]^+ \oplus \overline{\lie r}$, where 
$\lie g(\VV)[r]^+ = \lie g(\VV)[r] + \lie h$.
\begin{align*}
\begin{tikzpicture}
	\draw
		(0,0) rectangle (2,2); 
	\filldraw[color=black!20!white, draw=none]
		(0,2) -- (0.4,1.6) -- (0.4,0.4) -- (1.6,0.4) -- (2,0) -- (0,0) --
		cycle;
	\draw (1,-0.5) node {$\overline{\lie r}$}; 
\end{tikzpicture}
&&
\begin{tikzpicture}
	\draw
		(0,0) rectangle (2,2); 
	\filldraw[color=black!20!white, draw=none]
		(0.4,0.4) rectangle (1.6,1.6);
	\draw (0,2) -- (0.4,1.6);
	\draw (2,0) -- (1.6,0.4);
	\draw (1,-0.5) node {$\lie g(\VV)[r]^+$}; 
\end{tikzpicture}
&&
\begin{tikzpicture}
	\draw
		(0,0) rectangle (2,2); 
	\filldraw[color=black!20!white, draw=none]
		(0,2) -- (0.4,1.6) -- (1.6,1.6) -- (1.6,0.4) -- (2,0) -- (2,2) --
		cycle;
	\draw (1,-0.5) node {$\lie r$}; 
\end{tikzpicture}
\end{align*}
As $\lie g(\VV)[r]$-module, $M$ is isomorphic to a quotient of $\SS^{\bullet}
(\overline{\lie r}) \ot \SS^{\bullet}(\lie r) \ot \CC_a$, while $\lie r$ and
$\overline{\lie r}$ decompose as $\VV[r]^r \oplus \VV_*[r]^r \oplus 
\CC^{2r^2-r}$. By \cite{DCPS16}*{Lemma 4.1} there exists a finite dimensional 
subspace $X \subset \lie r$ such that $\SS^p(\lie r)$ is generated over $\lie 
g(\VV)[r]$ by $\SS^p(X)$ for all $p \geq 0$. On the other hand, since 
$M$ is integrable every element of $X$ acts nilpotently on $M$, and so 
$\SS^p(X)m = 0$, and hence $\SS^p(\lie r)m = 0$, for $p \gg 0$. By an analogous
reasoning, $\SS^p(\overline{\lie r})m = 0$, for $p \gg 0$.

It follows that $M$ is in fact isomorphic as $\lie g(\VV)[r]$-module to a 
quotient of $\SS^{\leq p}(\overline{\lie r}) \ot \SS^{\leq p}(\lie r) \ot 
\CC_a$ for $p \gg 0$. Since $\SS^{\leq p} (\overline{\lie r}) \ot \SS^{\leq p}
(\lie r)$ is a tensor $\lie g(\VV)[r]$-module it has finite length as $\lie 
g(\VV)[r]$-module. Thus $M$ has finite length over $\lie g(\VV)[r]$ and hence 
over $\lie g(\VV)$.
\end{proof}

We now turn back to the case $\lie g = \lie g(\VV^n)$.
We say that an $\lie l$-module is a \emph{tensor module} if it is a subquotient 
of the tensor algebra $T(\VV^{(1)} \oplus \VV^{(1)}_* \oplus \cdots \oplus 
\VV^{(n)} \oplus \VV^{(n)}_*)$. An indecomposable tensor module is isomorphic 
to an external tensor product $M_1 \boxtimes M_2 \boxtimes \cdots \boxtimes M_n$
with each $M_i$ a tensor $\lie g(\VV^{(k)})$-module. Simple tensor modules are 
parametrised by pairs $(\boldsymbol \lambda, \boldsymbol \mu)$ with $\boldsymbol
\lambda = (\lambda_1, \ldots, \lambda_n)$ and $\boldsymbol \mu = (\mu_1, 
\ldots, \mu_n)$ $n$-tuples of partitions, and 
\begin{align*}
	L(\boldsymbol \lambda, \boldsymbol \mu)
		&\cong L(\lambda_1, \mu_1) \boxtimes L(\lambda_2, \mu_2)
			\boxtimes \cdots \boxtimes L(\lambda_n, \mu_n).
\end{align*}
Tensor $\lie l$-modules have finite filtrations whose layers are simple tensor 
modules, and hence they have finite length. It is also clear they are 
integrable and satisfy the LAC by the characterisation of $\gl(\infty)$ tensor 
modules given above. 
\begin{defn}
The category $\TT_{\lie l}$ is the full subcategory of integrable $\lie 
g$-modules of finite length satisfying the LAC. For each $\chi \in \CC^n$ we 
define $\TT_{\lie l}^\chi$ to be the subcategory of $\TT_{\lie l}$ formed by 
modules of level $\chi$. 
\end{defn}
Tensor $\lie l$-modules belong to $\TT^0_{\lie l}$. Each $\TT_{\lie l}^\chi$ is 
a block of $\TT_{\lie l}$. Also $\CC_\chi \in \TT^\chi_{\lie l}$, and tensoring 
with $\CC_{\chi}$ gives an equivalence between $\TT_{\lie l}^0$ and $\TT_{\lie 
l}^\chi$.  
\begin{prop}
Let $M$ be an object of $\TT_{\lie l}$. 
\begin{enumerate}[(a)]
\item If $M$ is simple then it is isomorphic to $\CC_\chi \ot 
L(\boldsymbol \lambda, \boldsymbol \mu)$ for $\boldsymbol \lambda, \boldsymbol 
\mu \in \Part^n$.

\item If $M$ is a highest weight module with highest weight $\lambda$ and 
$\alpha$ is a positive root of $\lie l$ then $(\lambda, \alpha) \in \ZZ_{\geq 
0}$.
\end{enumerate}
\end{prop}
\begin{proof}
Since tensoring with a fixed one-dimensional module gives an equivalence 
between blocks of $\TT_{\lie l}$, to prove $(a)$ it is enough to prove that 
a simple object $L$ in $\TT^0_{\lie l}$ is a simple tensor module. 
Take $v \in L$, and set $L_k$ to be the $\lie g(\VV^{(k)})$-submodule of $L$
spanned by $v$. By definition $L_k$ is integrable and satisfies the LAC
with respect to $\lie g(\VV^{(k)})$, and so by Lemma \ref{lem:tensor-fg-fl}
it is a tensor $\lie g(\VV^{(k)})$-module. Since the $\lie g(\VV^{(k)})$ commute
with each other inside $\lie l$, there is a surjective map $L_1 \boxtimes L_2 
\boxtimes \cdots \boxtimes L_n \to L$, so $L$ is a quotient of a tensor module,
and hence a simple tensor module. To prove $(b)$, notice that if $M$ is a 
highest weight module in $\TT_{\lie l}$ then its unique simple quotient also 
lies in $\TT_{\lie l}$ and so it must be of the form $\CC_\chi \otimes L$ with 
$L$ a simple tensor module. The statement is easily proved for the highest 
weight of this simple module.
\end{proof}

\subsection{The category $\tilde \TT_{\lie l[r]}$}
Recall that 
weights in $\lie h^\circ$ are called eligible weights, and a weight is 
$r$-eligible if it is in the span of $\{\omega^{(k)}, \epsilon_i^{(k)} \mid k 
\in \interval n, i \in \pm \interval r\}$. 

Let $\lambda \in \lie h^*$. We denote by $\mathcal D(\lambda)$ the set of all 
weights $\lambda - \mu$ with $\mu \succeq 0$. Let $M$ be a weight $\lie 
g$-module. A weight $\lambda$ is said to be \emph{extremal} if $M_\lambda \neq 
0$ and $M_\mu = 0$ for all $\mu \succ \lambda$. If a module $M$ has finitely 
many extremal weights $\lambda_1, \ldots, \lambda_k$ then its support is 
contained in the union of the $\mathcal D(\lambda_i)$.

\begin{defn}
\label{defn:t-tilde}
The category $\tilde \TT_{\lie l[r]}$ is the full subcategory of $\Mod{(\lie 
l[r]^+, \lie h)}$ whose objects are modules $M$ satisfying the following 
conditions.
\begin{enumerate}[(i)]
\item $M$ has finitely many extremal weights. 
\item For each $\mu \in \lie h^*$ the $\lie l[r]$-module generated 
by $M_{\succeq \mu} = \bigoplus_{\nu \succeq \mu} M_\nu$ lies in 
$\TT_{\lie l[r]}$.
\end{enumerate}
\end{defn}

The following lemma is an easy consequence of the definition.
\begin{lem}
\label{lem:locally-tensor}
The category $\tilde \TT_{\lie l[r]}$ is closed under direct summands, finite 
direct sums and finite tensor products.
\end{lem}

\begin{ex}
\label{ex:vkl}
We have decompositions
\begin{align*}
\VV^{(k)} 
	&=V^{(k)}_{-,r} \oplus \VV^{(k)}[r] \oplus V^{(k)}_{+,r}
& \VV^{(k)}_* 
	&= \left(V^{(k)}_{-,r}\right)^* \oplus \VV^{(k)}_*[r] \oplus 
		\left(V^{(k)}_{+,r}\right)^*
\end{align*}
where $V^{(k)}_{\pm,r} = \langle v_i \ot e_k \mid i \in \pm \interval r 
\rangle$. Each of these modules is a tensor $\lie l[r]$-module: $\VV^{(k)}[r]$ 
is the natural representation of $\lie g(\VV^{(k)}[r])$, while $\VV^{(l)}_*[r]$ 
is the conatural representation of $\lie g(\VV^{(l)}[r])$, and the other four 
spaces are finite direct sums of $1$-dimensional $\lie h$-modules with trivial 
$\lie l[r]$ action. Since tensor modules are closed by finite direct sums, it 
follows that both are tensor modules over $\lie l[r]$. 

Now consider $\VV^{(k,l)} = \VV^{(k)} \ot \VV^{(l)}_*$. From the previous 
example we get a decomposition of this space into nine summands, illustrated in 
the picture below. 
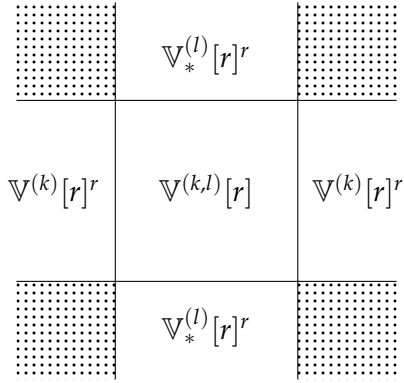
\begin{align*}
\begin{tikzpicture}
	\draw 
		(0,1.3) -- (5,1.3) (0,3.7) -- (5,3.7) (1.3,0) -- (1.3,5) 
			(3.7,0) -- (3.7,5);
	\filldraw [pattern=dots,draw=none] 
		(0,0) rectangle (1.3,1.3) (0,3.7) rectangle (1.3,5) 
		(3.7,0) rectangle (5,1.3) (3.7,3.7) rectangle (5,5);
	\draw (2.5,0.7) node {$\VV_*^{(l)}[r]^r$};
	\draw (2.5,4.3) node {$\VV_*^{(l)}[r]^r$};
	\draw (0.5,2.5) node {$\VV^{(k)}[r]^r$};
	\draw (4.5,2.5) node {$\VV^{(k)}[r]^r$};
	\draw (2.5,2.5) node {$\VV^{(k,l)}[r]$};
	\node [below right, text width=6cm,align=justify] at (6,5)
	{$\VV^{(k)} \ot \VV^{(l)}_*$ decomposes as $\lie l[r]$-module as a direct 
	sum of:
	\begin{itemize}
	\item[-] $\VV^{(k,l)}[r] = \VV^{(k)}[r] \ot \VV^{(l)}_*[r]$;
	\item[-] $2r$ copies of $\VV^{(k)}[r]$;
	\item[-] $2r$ copies of $\VV^{(l)}_*[r]$; 
	\item[-] a trivial module of dimension $4r^2$.
	\end{itemize}
	};
\end{tikzpicture}
\end{align*}

The fact that $\TT_{\lie l[r]}$ is closed by direct sums, direct summands and 
tensor products implies that $\SS^p(\VV^{(k,l)})$ is again in $\TT_{\lie l[r]}$. 
We claim that if $k > l$ then $\SS^\bullet(\VV^{(k,l)})$ is in $\tilde 
\TT_{\lie l[r]}$. Indeed, its support is contained in $\mathcal D(0)$, and 
furthermore for each $\lambda \in \lie h^\circ$ we have $\SS^\bullet
(\VV^{(k,l)})_{\succeq \lambda} \subset \bigoplus_{p=1}^{-\phi(\lambda)} 
\SS^{p}(\VV^{(k,l)})$, which lies in $\TT_{\lie l[r]}$; here $\phi$ is the map 
introduced in Example \ref{ex:gradings}. On the other hand if $k\leq l$ then 
the symmetric algebra is not in $\tilde \TT_{\lie l[r]}$, as it has no extremal 
weights.
\end{ex}

We will mostly be interested in studying $\lie g$-modules whose restriction to 
$\lie l[r]^+$ lies in $\tilde \TT_{\lie l[r]}$. We now show that such a module 
has LCS. The statement and proof are similar to \cite{DGK82}*{Proposition 3.2}, 
but we replace the dimension of $M_{\succeq \lambda}$, which could be infinite 
in our case, with the length of the $\lie l[r]$-module it spans. 
\begin{prop}
\label{prop:lcs}
Let $M$ be a $\lie g$-module lying in $\tilde \TT_{\lie l[r]}$. Then $M$ has 
LCS. 
\end{prop}
\begin{proof}
Fix $\mu$ in the support of $M$ and denote by $N(\mu)$ the $\lie l[r]$-submodule
spanned by $M_{\succeq \mu}$. The proof proceeds by induction on the length of 
$N(\mu)$, which we denote by $\ell$. If $\ell = 0$ then $0 \subset M$ is the 
desired composition series. In any other case we have $\mu \prec \lambda$ for 
some $\lambda$ which is maximal in the support of $M$. Taking $v \in M_\lambda$ 
we see that $V = U(\lie g) v$ is a highest weight module and hence has a unique 
maximal submodule $V'$. We then have a filtration
\begin{align}
0 \subset V' \subset V \subset M.
\end{align}
Since $v \in N(\mu)$ by hypothesis, $N'(\mu) = N(\mu) \cap V' \subsetneq N$ has 
length strictly less than $l$. Using this module and the induction hypothesis 
we know there exists an LCS at $\mu$ for $V'$. The same applies to $N(\mu)/
(N(\mu) \cap V) \subseteq M/V$, so $M/V$ also has a LCS at $\mu$. These LCS can 
be used to refine the filtration $V' \subset V$ into an LCS of $M$ at $\mu$.
\end{proof}

\subsection{The symmetric algebras of parabolic radicals}
\label{ss:sym-algs}
Recall that we have introduced for each $r \geq 0$ the subalgebra $\lie p(r)
= \lie l[r] + \lie b$. This subalgebra has a Levi-type decomposition $\lie p(r)
= \lie l[r]^+ \niplus \lie u(r)$, which in turn induces a decomposition 
$\lie g = \overline{\lie u}(r) \oplus \lie l[r]^+ \oplus \lie u(r)$.
The pictures of these subalgebras are given in subsection \ref{ss:visual-rep}.
We will now study the structure of $\SS^\bullet(\overline{\lie u}(r))$ as 
$\lie l[r]^+$-module, as this will be essential in the sequel.

Using the the decomposition of $V^{(k,l)}$ given in 
Example \ref{ex:vkl} it is easy to see that $\overline{\lie u}(r)$ is isomorphic
as $\lie l[r]$-module to
\begin{align*}
	Z \oplus 
		\left(\bigoplus_{k=1} \VV^{(k)}[r] \ot Y_k\right) \oplus
		\left(\bigoplus_{l=1} X_l \ot \VV^{(l)}_*[r]\right) \oplus
		\left(\bigoplus_{n \geq k>l \geq 1} \VV^{(k)}[r] \ot\VV^{(l)}_*[r]
			\right)
\end{align*}
for some trivial $\lie l[r]$-modules $X_k, Y_l, Z$, which we will study in more 
detail.

The above decomposition of $\overline{\lie u}(r)$ is in fact a decomposition 
into indecomposable $\overline{\lie b}_r \oplus \lie l[r]$-modules, where 
$\overline{\lie b}_r = \overline{\lie b} \cap \lie g_r$. Indeed, $Z$ 
coincides with the subalgebra $\overline{\lie n}_r = \overline{\lie n} \cap 
\lie g_r$. The modules $X_k$ and $Y_l$ can be described recursively as
\begin{align*}
X_{n} &= V_{+,r}^{(n)} 
	& X_{k-1} &= X_k \oplus V_{-,r}^{(k)}  \oplus V_{+,r}^{(k-1)}; \\
Y_{1} &= (V_{+,r}^{(1)})^* 
	& Y_{k+1} &= Y_k \oplus (V^{(k-1)}_{-,r})^* \oplus (V^{(k)}_{+,r})^*.
\end{align*}
It follows from this that $Y_k$ is isomorphic to the $\overline{\lie 
b}_r$-submodule of the conatural representation of $\lie g_r$ spanned by the 
unique vector of weight $-\epsilon^{(k)}_r$, and $X_l$ is isomorphic to the 
$\overline{\lie b}_r$-submodule of the natural representation of $\lie g_r$ 
spanned by the unique vector of weight $\epsilon^{(l)}_{-r}$. 

\begin{prop}
\label{prop:ur-computation}
The symmetric algebra $\SS^\bullet(\overline{\lie u}(r))$ is an object of 
$\tilde \TT_{\lie l[r]}$.
\end{prop}
\begin{proof}
To study the symmetric algebra $\SS^\bullet(\overline{\lie u}(r))$ we look at 
the symmetric algebra of each summand in the decomposition of $\overline{\lie 
u}(r)$ as $\overline{\lie b}_r \oplus \lie l[r]$-module separately.
\begin{align*}
	\SS^\bullet(\overline{\lie n}_r) 
		&= \bigoplus_{k \geq 0} \SS^k \left(\overline{\lie n}_r\right) \\
	\SS^\bullet\left(\VV^{(k)}[r] \ot Y_{k}\right) 
		&= \bigoplus_{\nu_k} \Schur_{\nu_k}\left(\VV^{(k)}[r]\right)
			\otimes \Schur_{\nu_k}\left(Y_{k}\right) & (k = 1, \ldots, n)\\
	\SS^\bullet\left(X_{l} \otimes \VV^{(l)}_*[r] \right) 
		&= \bigoplus_{\gamma_l} \Schur_{\gamma_i}\left(X_{l}\right) 
			\otimes \Schur_{\gamma_l}\left(\VV^{(l)}_*[r]\right) 
				& (l = 1, \ldots, n)\\
	\SS^\bullet\left(\VV^{(k)}[r] \ot \VV^{(l)}_*[r]\right)
		&= \bigoplus_{\rho_{k,l}} \Schur_{\rho_{k,l}}\left(\VV^{(k)}[r]\right) 
			\ot \Schur_{\rho_{k,l}}\left(\VV^{(l)}_*[r]\right) 
				&(1 \leq l < k \leq n)
\end{align*}
By Lemma \ref{lem:locally-tensor} it is enough to see that each symmetric 
algebra lies in $\tilde \TT_{\lie l[r]}$, and in each case we will show that if 
$\mu \in \lie h^*$ is in the support then it is in the support of finitely 
many of the indecomposable summands. 

\emph{Case 1: $\SS^\bullet(\overline{\lie n}_r)$}. If $\mu$ is any weight in 
the support then $\SS^\bullet(\overline{\lie n}_r)_{\succeq \mu}$ must be 
finite dimensional. In particular it can only intersect $\SS^k(\overline{\lie 
n}_r)$ for finitely many $k \in \ZZ_{>0}$.

\emph{Case 2: $\SS^\bullet\left(\VV^{(k)}[r] \ot Y_k\right)$}. Fix a weight
$\mu$ in the support. Recall that $\mu|_r$ and $\mu[r]$ denote the restriction 
of $\mu$ to $\lie h_r$ and $\lie h[r]$, respectively. Then
\begin{align*}
	\left(\Schur_{\nu}\left(\VV^{(k)}[r]\right) 
		\ot \Schur_{\nu}(Y_k)\right)_\mu
		&= \Schur_{\nu}\left(\VV^{(k)}[r]\right)_{\mu[r]} 
		\ot \Schur_{\nu}(Y_k)_{\mu|_r}.
\end{align*}
Since the support of $Y_{k}$ consists entirely of negative weights, 
$\Schur_{\nu}(Y_{k})_{\succeq \mu|_r} = 0$ whenever $|\nu|$ is larger than the 
height of $\mu|_r$ as a weight of $\overline{\lie b}_r$. 

\emph{Case 3: $\SS^\bullet\left(X_k \ot \VV^{(k)}_*[r]\right)$}. As in case 
$2$ we have a decomposition
\begin{align*}
	\left(\Schur_{\nu}\left(\VV^{(l)}_*[r])\right) 
		\ot \Schur_{\nu}(X_{l})\right)_\mu
		&= \Schur_{\nu}\left(\VV^{(l)}[r]_*\right)_{\mu[r]} 
		\ot \Schur_{\nu}(X_{l})_{\mu|_r}.
\end{align*}
Now the support of $X_l$ consists entirely of positive weights, and the 
argument is analogous to the previous case.

\emph{Case 4: $\SS^\bullet\left(\VV^{(k)}[r] \ot \VV^{(l)}_*[r]\right)$}.
This was done in example \ref{ex:vkl}.
\end{proof}
Notice that while $\SS^p(\lie u(r))$ is a tensor $\lie l[r]$-module for any 
$p$, $\SS^\bullet(\lie u(r))$ is \emph{not} in $\tilde \TT_{\lie l[r]}$.

\subsection{Large annihilator duality}
\label{ss:lac}
The categories $\TT_{\lie g}$ and $\TT_{\lie l}$ are not closed under 
semisimple duals. For example $\lie h^* = \lie g^\vee_0$ contains uncountably 
many vectors that do not satisfy the LAC, so $\lie g^\vee$ is not in 
$\TT_{\lie g}$. For this reason we introduce the \emph{large annihilator dual} 
of a $\lie l$-module, which we denote by $\dual M$ and set to be $\Phi_{\lie l}
(M^\vee)$. We use the same notation when $M$ is a $\lie g$-module, seen as 
$\lie l$-module by restriction. The assignation $M \mapsto \dual M$ is  
left exact, but since it is a composition of continuous functors it sends 
colimits in $\Mod{(\lie g, \lie h)}$ into limits. 

Given two $n$-tuples of partitions $\boldsymbol \lambda, \boldsymbol \mu$ we set
\begin{align*}
I(\boldsymbol \lambda, \boldsymbol \mu)
	&= I(\lambda_1, \mu_1) \boxtimes \cdots \boxtimes I(\lambda_n, \mu_n).
\end{align*}
The socle filtration of these modules can be deduced from the socle filtration 
of the $I(\lambda_i, \mu_i)$, and as a direct consequence we see that there is 
a nonzero morphism $I(\boldsymbol \lambda, \boldsymbol\mu) \to \CC$ if and only 
if $\boldsymbol \lambda = \boldsymbol \mu$, and in that case this map is unique.

\begin{lem}
\label{lem:la-duality}
The large annihilator dual of $I(\boldsymbol \lambda, \boldsymbol \mu)$ is
isomorphic to
\begin{align*}
\bigoplus_{\boldsymbol \alpha, \boldsymbol \beta}
\bigoplus_{\boldsymbol \gamma} \left(
	\prod_k c_{\alpha_k, \gamma_k}^{\lambda_k} c_{\beta_k, \gamma_k}^{\mu_k}
\right) I(\boldsymbol \alpha, \boldsymbol \beta),
\end{align*}
where the sum runs over all $n$-tuples of partitions $\boldsymbol \alpha, 
\boldsymbol \beta$ and $\boldsymbol \gamma$. In particular $\dual{I(\boldsymbol 
\lambda, \boldsymbol \mu)}$ lies in $\TT_{\lie l}^0$.
\end{lem}
\begin{proof}
Set $\lie l_r = \lie l \cap \lie g_r$. There exist decompositions of
$\lie l_r \oplus \lie l[r]$-modules
\begin{align*}
\VV^{(k)} &= V^{(k)}_r \oplus \VV^{(k)}[r]; &
	\VV^{(k)}_* &= \left(V^{(k)}_r\right)^* \oplus \VV^{(k)}_*[r].
\end{align*}
Using the decomposition of the image of a direct sum by a Schur functor we get
\begin{align*}
	I(\lambda,\mu)
		&= \bigoplus_{\alpha, \beta,\gamma,\delta}
			c_{\alpha,\gamma}^\lambda c_{\beta,\delta}^\mu
			\Schur_\alpha(V^{(k)}_r) 
				\otimes 
			\Schur_\beta\left( \left(V^{(k)}_r \right)^* \right) \boxtimes
			I(\gamma,\delta)[r],
\end{align*}
where $I(\gamma,\delta)[r] = \Schur_\gamma(\VV^{(k)}[r]) \otimes 
\Schur_\delta(\VV^{(k)}_*[r])$, and this is an isomorphism of $\lie l_r \oplus 
\lie l[r]$-modules. 

Notice that twisting the modules $\VV^{(k)}$ and $\VV^{(k)}_*$ by the 
automorphism $-\tau$ produces isomorphic $\lie l$-modules. Since twisting by 
$-\tau$ commutes with Schur functors, it follows that ${}^{-\tau}I(\boldsymbol 
\lambda, \boldsymbol \mu) \cong I(\boldsymbol \lambda, \boldsymbol \mu)$. 
Computing the semisimple duals we get
\begin{align*}
	(I(\boldsymbol \lambda, \boldsymbol \mu))^\vee
		&\cong \ssHom_{\CC}
			({}^{-\tau}I(\boldsymbol \lambda, \boldsymbol \mu),\CC)
		\cong \ssHom_{\CC}
			(I(\boldsymbol \lambda, \boldsymbol \mu), \CC).
\end{align*}
Now writing $I(\boldsymbol \lambda, \boldsymbol \mu)$ as the tensor product of
the $I(\lambda_i, \mu_i)$ and using the above decomposition, we get
\begin{align*}
(I(\boldsymbol \lambda, \boldsymbol \mu))^\vee
	&\cong 
		\ssHom_{\CC} \left(
			\bigotimes_{k=1}^n
	T_{\alpha_k,\beta_k,\gamma_k,\delta_k}^{(k)} \boxtimes
		I(\gamma_k, \delta_k),\CC
	\right)
\end{align*}
where
\begin{align*}
T_{\alpha,\beta,\gamma,\delta}^{(k)} &= 
	\bigoplus_{\gamma,\delta}
	c_{\alpha,\delta}^{\lambda} 
		c_{\beta,\gamma}^{\mu}
			\left(\Schur_\alpha(V_r^{(k)}) \otimes 
				\Schur_{\beta}((V_r^{(k)})^*)
				\right).
\end{align*}
This last module is a semisimple finite dimensional $\lie l_r$-module, and 
hence isomorphic to its semisimple dual. Using the fact that duals 
distribute over tensor products if one of the factors is finite 
dimensional, we obtain an isomorphism of $\lie l_r \oplus \lie l[r]$-modules
\begin{align*}
I(\boldsymbol \lambda, \boldsymbol \mu)
	\cong \bigoplus_{\boldsymbol \alpha, \boldsymbol \beta, 
		\boldsymbol \gamma, \boldsymbol \delta}
			\left(
				\bigotimes_{k=1}^n T_{\alpha_k,\beta_k,\gamma_k,\delta_k}^{(k)}
			\right)
	\boxtimes \ssHom_{\CC} \left(
		I(\boldsymbol \gamma, \boldsymbol \delta)[r], \CC
	\right).
\end{align*}
Now to compute the image of this module by $\Phi_r$ only the $\lie l[r]$-module
structure is relevant, and 
\begin{align*}
\Phi_r\left(\ssHom_{\CC} \left(
		I(\boldsymbol \gamma, \boldsymbol \delta)[r], \CC
	\right)\right)
	&\cong
	\ssHom_{\lie l[r]} \left(
		I(\boldsymbol \gamma, \boldsymbol \delta)[r], \CC
	\right)
	\cong
	\ssHom_{\lie l} \left(
		I(\boldsymbol \gamma, \boldsymbol \delta), \CC
	\right).
\end{align*}
As mentioned in the preamble this is the zero vector space unless $\boldsymbol
\gamma = \boldsymbol \delta$, in which case it is one-dimensional. Thus 
$\Phi_r(I(\boldsymbol \lambda, \boldsymbol \mu)^\vee)$ is isomorphic to 
\begin{align*}
\bigoplus_{\boldsymbol \alpha, \boldsymbol \beta, 
		\boldsymbol \gamma}
			\left(
				\bigotimes_{k=1}^n T_{\alpha_k,\beta_k,\gamma_k,\gamma_k}^{(k)}
			\right)
&=
\bigoplus_{\boldsymbol \alpha, \boldsymbol \beta, 
		\boldsymbol \gamma}
\bigotimes_{k=1}^n
c_{\alpha_k,\gamma_k}^{\lambda_k} 
		c_{\beta_k,\gamma_k}^{\mu_k}
			\left(\Schur_{\alpha_k}(V_r^{(k)}) \otimes 
				\Schur_{\beta_k}((V_r^{(k)})^*)
				\right).
\end{align*}
Taking the limit as $r$ goes to infinity, we see that $\dual{I(\boldsymbol 
\lambda, \boldsymbol \mu)}$ is isomorphic to
\begin{align*}
\bigoplus_{\boldsymbol \alpha, \boldsymbol \beta, 
		\boldsymbol \gamma}
\bigotimes_{k=1}^n
c_{\alpha_k,\gamma_k}^{\lambda_k} 
		c_{\beta_k,\gamma_k}^{\mu_k}
			\left(\Schur_{\alpha_k}(\VV^{(k)}) \otimes 
				\Schur_{\beta_k}(\VV_*^{(k)})
				\right),
\end{align*}
and this precisely the module in the statement.
\end{proof}
The following is an immediate consequence of the last two results.
\begin{prop}
\label{p:ur-ladual}
The large annihilator dual of $\SS^\bullet(\overline{\lie u}(r))$ lies in 
$\tilde \TT_{\lie l[r]}$.
\end{prop}

\section{Category $\CAT{}{}$: first definitions}
\label{s:ola}
As in the previous section we fix $n \in \ZZ_{>0}$ and denote by $\lie g$ the
Lie algebra $\lie g(\VV^n)$. We will also continue to omit $\VV^n$ from the 
notation for subalgebras of $\lie g$. 

\subsection{Introduction to $\CAT{}{}$}
We now introduce a category of representations of $\lie g$ that serves as an 
analogue of category $\mathcal O$. The definition is analogous to the usual 
definition of category $\mathcal O$ for the reductive Lie algebra $\lie{gl}
(r,\CC)$ but, since $U(\lie{gl}(\infty))$ is not noetherian, we need to replace
finite generation with the LAC with respect to $\lie l$.

\begin{defn}
The category $\CAT{\lie l}{\lie g}$ is the full subcategory of $\Mod{\lie g}$ 
whose objects are the $\lie g$-modules $M$ satisfying the following conditions.
\begin{enumerate}[(i)]
\item $M$ is $\lie h$-semisimple.
\item $M$ is $\lie n$-torsion.
\item $M$ satisfies the LAC with respect to $\lie l$.
\end{enumerate}
\end{defn}

We often write $\CAT{}{}$ for $\CAT{\lie l}{\lie g}$. If $M$ is an object
of $\CAT{}{}$ then so is any subquotient of $M$. It follows from the definition 
that the support of an object $M$ in $\CAT{}{}$ must be contained in $\lie 
h^\circ$. The following lemma shows that finitely generated objects in 
$\CAT{}{}$ have LCS and hence well-defined Jordan-Holder multiplicities. Recall 
from Example \ref{ex:gradings} that there is a map $\psi: \lie h^\circ \to \CC$ 
that sends the finite roots to $0$, induces a $\ZZ$-grading on $\lie g$, and 
turns a weight module $M$ into a graded module $M^\psi$.

\begin{lem}
\label{lem:ola-generator}
Let $M$ be any $\lie g$-module and let $v \in M$ be an $\lie h$-semisimple and 
$\lie n$-torsion vector satisfying the LAC with respect to $\lie l$. Then the 
submodule $N = U(\lie g)v$ lies in $\CAT{}{}$, has LCS, and $N^\psi_z = 0$ for 
$z \gg 0$.
\end{lem}
\begin{proof}
Take $r \gg 0$ such that $\CC v$ is a trivial $\lie l[r]$-module. Using the PBW 
theorem we have a surjective map of $\lie l[r]$-modules
\begin{align*}
p:\SS^\bullet(\overline{\lie u}(r)) \ot \SS^\bullet(\lie u(r)) \ot 
	\SS^\bullet(\lie l[r]^+) \ot \CC_\lambda &\to N \\
	u \ot u' \ot l \ot 1_\lambda &\mapsto uu'lv.
\end{align*}
Since $v$ is $\lie n$-nilpotent and satisfies the LAC we have $\SS^\bullet(\lie 
l[r]^+)v = \CC v$ and $\SS^k(\lie u(r))v = 0$ for $k \gg 0$. Thus the 
restriction
\begin{align*}
p':\SS^\bullet(\overline{\lie u}(r)) \ot \SS^{\leq k}(\lie u(r)) \ot 
	\CC_\lambda &\to N 
\end{align*}
is surjective. By Propositions \ref{prop:lcs} and \ref{prop:ur-computation} the 
domain of this last map lies in $\tilde \TT_{\lie l[r]}$ and 
has LCS. Since $\overline{\lie u(r)}^\psi$ has a right-bounded grading, the 
same holds for the domain of $p'$ and hence for $N$. By the PBW theorem
\begin{align*}
\lie n \cdot \SS^i(\overline{\lie u}(r)) \ot \SS^j(\lie u(r))
	\subset \bigoplus_{i',j',t'}
		\SS^{i'}(\overline{\lie u}(r)) \ot \SS^{j'}(\lie u(r)) 
			\ot \SS^{t'}(\lie l[r]^+)
\end{align*}
where $i' + j' + t' = i + j + 1$ and $i' \leq i$. Induction on $i$ shows that
$p'(\SS^i(\overline{\lie u}(r)) \ot \SS^j(\lie u(r)) \ot \CC_\lambda)$ is 
annihilated by a large enough power of $\lie n$, and so $N$ is $\lie n$-torsion.
\end{proof}

This lemma has a very useful consequence: every module in $\Mod{(\lie g, \lie 
h)}$ has a largest submodule contained in $\CAT{}{}$, namely the submodule 
spanned by its $\lie n$-torsion elements satisfying the LAC with respect to 
$\lie l$. Categorically, this means that the inclusion functor of $\CAT{}{}$ in
$\Mod{(\lie g, \lie h)}$ has a right adjoint $\LAprojector$. We will come back
to this observation later on, when we look at categorical properties of 
$\CAT{}{}$.

\subsection{Simple objects in $\CAT{}{}$}
Let $\lambda \in \lie h^*$. The Verma module $M(\lambda)$ does not belong to 
$\CAT{}{}$ since the highest weight vector does not satisfy the
LAC. However, if $\lambda$ is $r$-eligible there exists a $1$-dimensional 
$\lie l [r]^+$-module $\CC_\lambda$, which we can inflate it to a $\lie 
p(r)$-module setting the action of $\lie u(r)$ to be zero. Set $M_r(\lambda) = 
\Ind_{\lie p(r)}^{\lie g} \CC_\lambda$, which is clearly a highest weight 
module of weight $\lambda$. 

Since the highest weight vector of $M_r(\lambda)$ satisfies the LAC, Lemma 
\ref{lem:ola-generator} implies that $M_r(\lambda)$ lies in $\CAT{}{}$. Notice 
also that we have surjective maps $M_{r+1}(\lambda) \to M_r(\lambda)$, that 
given a weight $\mu$ in the support the restriction $M_{r+1}(\lambda)_\mu \to 
M_r(\lambda)_\mu$ is an isomorphism for $r\gg 0$, and that $M(\lambda)$ is the 
inverse limit of this system. This shows in particular that $\CAT{}{}$ is not 
closed under inverse limits. With these parabolic Verma modules in hand, we are 
ready to prove our first result.

\begin{thm}
\label{thm:simples}
The simple objects in $\CAT{}{}$ are precisely the highest weight simple modules
whose highest weight is in $\lie h^\circ$.
\end{thm}
\begin{proof}
Suppose $L$ is a simple object in $\CAT{}{}$. Since $L$ is $\lie n$-torsion it 
has a highest weight vector $v$ of weight $\lambda \in \lie h^\circ$, so 
$L \cong L(\lambda)$. On the other hand, if $\lambda$ is an $r$-eligible 
weight then $M_r(\lambda)$ lies in $\CAT{}{}$ and so does its unique simple 
quotient $L(\lambda)$.
\end{proof}

\begin{rmk}
A more subtle difference between the definition of $\CAT{}{}$ and that of 
category $\mathcal O$ for finite dimensional reductive Lie algebras is that we 
ask for $M$ to be $\lie n$-torsion and not just locally $\lie n$-nilpotent. 
In the finite-dimensional case, and even in the case $n = 1$, these two 
conditions are equivalent (see \cite{PS19}*{Proposition 4.2} for 
$\lie{sl}(\infty)$, the proof is the same of $\gl(\infty)$). 

This is no longer true as soon as $n \geq 2$. Indeed, when $n=2$ the simple 
\emph{lowest} weight module with lowest weight $\lambda = - \omega^{(1)} + 
\omega^{(2)}$ is the limit of the simple finite dimensional lowest weight 
modules $\tilde L(\lambda|_k)$, where each embeds in the next by sending the 
lowest weight vector to the lowest weight vector. Thus $\tilde L(\lambda) = 
\varinjlim \tilde L(\lambda|_k)$ is generated by a weight vector satisfying the 
LAC, and the construction shows that $\tilde L(\lambda)$ is locally $\lie 
n$-nilpotent, but not $\lie n$-torsion. Notice that this module can not be a
highest weight module: indeed, if it had a highest weight vector $v$ then $v$ 
would belong to $\tilde L(\lambda|_k)$ for all $k \gg 0$ and be a highest 
weight vector, but the highest weight vectors of $\tilde L(\lambda|_k)$ are 
never sent to highest weight vectors of $\tilde L(\lambda|_{k+1})$.
\end{rmk}

\subsection{Highest weight modules in $\CAT{}{}$}
As a consequence of Lemma \ref{lem:ola-generator} a highest weight module $M$ 
in $\CAT{}{}$ has finite Jordan-Holder multiplicities, and furthermore if 
the highest weight vector generates a $1$-dimensional $\lie l[r]$-module then 
$M$ lies in $\tilde \TT_{\lie l[r]}$. We will now show that highest 
weight modules have finite length. 

Recall that we denote by $\lie s$ and $\lie m$ the subalgebras of $\lie g$ 
spanned by root spaces corresponding to finite and positive infinite roots, 
respectively, and that $\lie s = \lie g^\psi_0$ and $\lie m = \lie g^\psi_{>0}$.
If $M$ is an object of $\CAT{}{}$ we can see it as a $\ZZ$-graded module through
the map $\psi$, and each homogeneous component is a $\lie s$-module. If the 
grading on $M$ is right-bounded, for example if $M$ is finitely generated, we 
will denote by $M^+$ the top nonzero homogeneous component.

We have also set $\lie q = \lie s \oplus \lie m$. An $\lie s$-module can 
be inflated into a $\lie q$-module by imposing a trivial $\lie m$-action. The 
following result shows that the problem of computing Jordan-Holder 
multiplicities of a highest weight module in $\CAT{}{}$ reduces to computing 
Jordan-Holder multiplicities of highest weight $\lie s$-modules in $\overline 
{\mathcal O}_{\lie s}$. This is very useful since weight components of finite 
length $\lie s$-modules are finite dimensional, and hence characters can be 
given in terms of these dimensions. 
\begin{prop}
\label{prop:hwm-properties}
Let $M \in \CAT{}{}$ be a highest weight module of highest weight $\lambda$. 
\begin{enumerate}[(i)]
\item If $N \subset M$ is a nontrivial module then $N^+ = N \cap M^+$.

\item $M$ is simple if and only if $M^+$ is a simple $\lie s$-module.

\item $M = \Ind^{\lie g}_{\lie q} \mathcal I_{\lie s}^{\lie q} M^+$. 

\item If $\mu$ is an eligible weight and $[M:L(\mu)] \neq 0$ then $\lambda
- \mu$ is finite, and furthermore $[M:L(\mu)] = [M^+:L_{\lie s}(\mu)] \leq 
m(\lambda,\mu)$.
\end{enumerate}
\end{prop}
\begin{proof}
Denote by $v$ the highest weight vector of $M$. If $N$ is a nontrivial 
submodule of $M$ then it contains a highest weight vector, say $w$ of weight 
$\mu$. For each $r \geq 0$ denote by $M_r$ the $\lie g_r$-module generated 
by $v$. Then $M_r$ is a highest weight $\lie g_r$-module and for $r$ large 
enough $w \in M_r$, and it is a highest weight vector. Thus $\mu|_r$ and 
$\lambda|_r$ are linked for all large $r$. Lemma \ref{lem:linked-weights} tells
us that $\lambda - \mu$ must then be a finite root, and so $w \in N \cap M^+ 
\neq \emptyset$ and $N^+ = N \cap M^+$. This proves the first item.

If $M$ is simple then $M^+$ is a $U(\lie g)^\psi_0$ simple module. Using the 
PBW theorem we have a decomposition $U(\lie g)^\psi_0 = \bigoplus_{k \in \NN} 
U(\overline{\lie m})^\psi_{-k}U(\lie s) U(\lie m)^\psi_k$, and since $\lie m$ 
acts trivially on $M^+$ it follows that $M^+$ is a simple $U(\lie s)$-module. 
Conversely, if $M^+$ is simple then any nonzero submodule $N \subset M$ must 
have $N^+ = M^+$ by the first item, so it must contain the highest weight 
vector. Thus $N = M$ and $M$ is simple.

Denote by $K$ be the kernel of the natural map $\Ind_{\lie q}^{\lie g} M^+ \to 
M$. Then by construction $K \cap M^+ = 0$, and the first item implies $K = 0$. 
In particular this shows that $L(\mu) = \Ind_{\lie q}^{\lie g} \mathcal 
I_{\lie s}^{\lie q} L_{\lie s}(\mu)$.
Finally, the module $M^+$ is a highest weight $\lie s$-module and its highest 
weight $\lambda$ is almost dominant since this holds for all eligible 
weights. Thus $M_{\lie s}(\lambda)$ has a composition series and if we 
apply the exact functor $\Ind_{\lie q}^{\lie g} \circ \mathcal I_{\lie s}^{\lie 
q}$ to this filtration we get a filtration of $M$. By the previous item the 
layers of this filtration are simple modules, and hence it is again a 
composition series and we can use it to compute
\begin{align*}
[M: L(\mu)] &= [M^+:L_{\lie s}(\mu)].
\end{align*}
Finally, by the universal property of Verma modules $M^+$ is a quotient of 
$M_{\lie s}(\lambda)$, so $[M^+:L_{\lie s}(\mu)] \leq m(\lambda,\mu)$.
\end{proof}

We will now show that parabolic Verma modules have finite length, and compute 
their Jordan-Holder multiplicities in terms of multiplicities of Verma modules 
over $\lie g_k$.
\begin{cor}
\label{cor:parabolic-verma-jh}
Let $\lambda, \mu \in \lie h^\circ$ with $\lambda$ $r$-eligible. If $L(\mu)$
is a simple constituent of $M_r(\lambda)$ then the following hold.
\begin{enumerate}[(i)]
\item $\mu$ lies in the dot orbit of $\lambda$ by $\mathcal W(\lie s_{r+1})$.
\item If $\mu$ is $r$-eligible then $[M_r(\lambda): L(\mu)] = m(\lambda,\mu)$.
\end{enumerate} 
In particular $M_r(\lambda)$ has finite length.
\end{cor}
\begin{proof}
We have already seen in Proposition \ref{prop:hwm-properties} that $\mu$ and
$\lambda$ must be linked. Since $M_r(\lambda)$ is an object of $\tilde 
\TT_{\lie l[r]}$ so is $L(\mu)$, and in particular $\mu|_r$ must be the highest 
weight of a highest weight module in $\TT_{\lie l[r]}$. Thus $(\lambda, 
\alpha)$ and $(\mu, \alpha)$ are positive for any finite root $\alpha$ of 
$\lie l[r]$, and Lemma \ref{lem:linked-weights} implies that 
$\mu = \sigma \cdot \lambda$ for $\sigma \in W(\lie s_{r+1})$. In particular 
$\mu$ belongs to a finite set, so $M_r(\lambda)$ has finite length.

We also know from Proposition \ref{prop:hwm-properties} that $[M_r(\lambda):
L(\mu)] = [M_r(\lambda)^+, L_{\lie s}(\mu)]$. The surjective map of $\lie 
s$-modules $M_{\lie s}(\lambda) \to M_r(\lambda)^+$ restricts to a bijection of
the weight components $\succeq \mu$, and so
\begin{align*}
\dim M_r(\lambda)^+_\mu 
    &=\sum_{\nu \succeq \mu} [M_r(\lambda)^+:L_{\lie s}(\nu)] 
		\dim L_{\lie s}(\nu)_\mu \\
	&= \dim M_{\lie s}(\lambda)_{\mu}
	= \sum_{\nu \succeq \mu} m(\lambda,\nu) \dim L_{\lie s}(\nu)_\mu
\end{align*}
Thus $[M_r(\lambda)^+:L_{\lie s}(\mu)] = m(\lambda,\mu)$.
\end{proof}

\subsection{The structure of a general object in $\CAT{}{}$}
We now turn to more general objects of $\CAT{}{}$. First we show that finitely 
generated objects in $\CAT{}{}$ must have finite length. 
\begin{prop}
\label{prop:fg-fl}
Let $M$ be an object of $\CAT{}{}$. The following are equivalent.
\begin{enumerate}[(a)]
 	\item $M$ is finitely generated.
 	\item $M$ has a finite filtration whose layers are highest weight modules.
 	\item $M$ has finite length.
\end{enumerate}
\end{prop}
\begin{proof}
To prove $(a) \Rightarrow (b)$ suppose first that $M$ is cyclic, and its 
generator is a weight vector $v$ of weight $\lambda$ such that $\lie l[r]' 
v = 0$. As seen in Lemma  \ref{lem:ola-generator} $M$ is in $\tilde \TT_{\lie 
l[r]}$, so we can proceed by induction on the length of the $\lie l[r]$-module 
spanned by $M_{\succeq \lambda}$. This submodule 
contains a highest weight vector, say $w$, and we set $N = U(\lie g)w$. Clearly 
$N$ is a highest weight module and $M/N$ has a filtration by highest weight 
modules by hypothesis, so the same holds for $M$. The general case now follows 
by induction on the number of generators of $M$. The implication $(b) 
\Rightarrow (c)$ follows from Corollary \ref{cor:parabolic-verma-jh}, and $(c) 
\Rightarrow (a)$ is obvious.
\end{proof}

We now give a general approach to compute the Jordan-Holder multiplicities of 
an arbitrary object of $\CAT{}{}$. We will use the following tool.

\begin{defn}
\label{defn:psi-filtration}
Let $M \in \CAT{}{}$. For each $i \in \ZZ$ we set $\FF_iM$ to be the submodule 
of $M$ generated by $\bigoplus_{j \geq i} M_j^\psi$. The family $\left\{\FF_i 
M  \mid i \in \ZZ\right\}$ is the \emph{$\psi$-filtration} of $M$. We also 
define $\mathcal S_i M$ to be the $\lie s$-module $(\FF_i M / \FF_{i-1} M)_i$,
which is the top component of the corresponding layer of the filtration.
\end{defn}

By construction the layers of the $\psi$-filtration of any module are  
generated by their top degree component, and hence their multiplicities can be computed using Proposition \ref{prop:hwm-properties}. 

\begin{prop}
\label{prop:s-filtration}
Let $M$ be an object in $\CAT{}{}$ with LCS and let $\lambda$ be an eligible 
weight with $\psi(\lambda) = p$. Then $[M:L(\lambda)] = [\mathcal S_p M: 
L_{\lie s}(\lambda)]$. 
\end{prop}
\begin{proof}
It is enough to show the result for finitely generated objects, and by 
Proposition \ref{prop:fg-fl} it is enough to show it for finite length objects. 
Let $M$ be a finite length object, and take $q \in \ZZ$ such that $M^+ = 
M^\psi_q$. A simple induction on $p-q$ shows that $[\mathcal S_p M: L_{\lie s}
(\lambda)]$ is an additive function on $M$, with the base case a consequence of 
Proposition \ref{prop:hwm-properties}. Since both $[M:L(\lambda)]$ and 
$[\mathcal S_p M: L_{\lie s}(\lambda)]$ are additive functions on $M$ it is 
enough to show they are equal when $M$ is simple, which again follows from 
Proposition \ref{prop:hwm-properties}. 
\end{proof}

\subsection{Categorical properties of $\CAT{}{}$}
We now focus on the general categorial properties of $\CAT{}{}$.
Let $M$ be a $\lie h$-semisimple $\lie g$-module. By Lemma 
\ref{lem:ola-generator} the submodule spanned by all its $\lie n$-torsion 
vectors satisfying the large annihilator condition is an object of $\CAT{}{}$,
and is in fact the largest submodule of $M$ lying in $\CAT{}{}$. We thus
have a diagram of functors, where each arrow from left to right is an embedding
of categories and each right to left arrow is a left adjoint
\begin{align*}
\xymatrix{
	\CAT{\lie l}{\lie g} \ar@<-1ex>[r]
	& \Mod{(\lie g, \lie h)}_{\mathsf{LA}}^{\lie l} 
		\ar@<-1ex>[r] \ar@<-1ex>[l]_-{\Gamma_{\lie n}}
	& \Mod{(\lie g, \lie h)}
		\ar@<-1ex>[r] \ar@<-1ex>[l]_-{\Phi}
	&\Mod{\lie g}
		\ar@<-1ex>[l]_-{\Gamma_{\lie h}}
}
\end{align*}
It follows that we have a functor $\LAprojector = \Gamma_{\lie n} \circ \Phi
=\Phi \circ \Gamma_{\lie n}: \Mod{(\lie g, \lie h)} \to \CAT{}{}$, which is  
right adjoint to the exact inclusion functor $\CAT{}{} \to \Mod{(\lie g, \lie 
h)}$. In particular it preserves direct limits and sends injectives to 
injectives.

Recall that an abelian category $\mathcal A$ is locally artinian if every 
object is the limit of its finite length objects. Also, $\mathcal A$ has the 
Grothendieck property if it has direct limits, and for every object $M$, every 
subobject $N \subset M$ and every directed family of subobjects $(A_\alpha)_{
\alpha \in I}$ of $M$ it holds that $N \cap \varinjlim A_\alpha = \varinjlim N 
\cap A_\alpha$. The following result is an easy consequence of the properties 
of $\LAprojector$.
\begin{thm}
Category $\CAT{}{}$ is a locally artinian category with direct limits, enough 
injective objects, and the Grothendieck property.
\end{thm}

\section{Category $\CAT{}{}$: Standard objects}
\label{s:standard}

In this section we will introduce the standard objects of $\CAT{}{}$. We then
compute their simple multiplicities through their $\psi$-filtrations.

\subsection{Large-annihilator dual Verma modules}
We now introduce the modules that will play the role of standard objects on 
$\CAT{}{}$. In the finite dimensional case this role is played by the 
semisimple duals of Verma modules, so it is natural to consider the ``best 
approximation'' to these modules in $\CAT{}{}$.
\begin{defn}
For every $\lambda \in \lie h^\circ$ we set $A(\lambda) = \dual{M(\lambda)}$. 
\end{defn}
The module $A(\lambda)$ can be hard to grasp. For example, it is not clear 
at first sight that it lies in $\CAT{}{}$. As a first approximation, set 
$A_r(\lambda)$ to be the LA dual of the parabolic Verma $M_r(\lambda)$. The 
Verma module $M(\lambda)$ is the inverse limit of the $M_r(\lambda)$, and since 
LA duality sends inverse limits to direct limits, $A(\lambda) = \varinjlim 
A_r(\lambda)$. Also, the natural maps $A_r(\lambda) \to A(\lambda)$ are 
injective, and for any $\mu$ in the support of $A(\lambda)$ and $r \gg 0$ the 
map $A(\lambda)_{\succeq\mu} \to A_r(\lambda)_{\succeq\mu}$; this follows from 
the dual facts for $M_r(\lambda)$ and $M(\lambda)$. This approach 
allows us to prove that the $A(\lambda)$ are indeed in $\CAT{}{}$.

\begin{prop}
\label{prop:a-in-ola}
For every eligible weight $\lambda$ and every $r \geq 0$ the modules 
$A(\lambda), A_r(\lambda)$ lie in $\CAT{}{}$ and have LCS. Furthermore,
we have $[A(\lambda): L(\mu)] = [A_r(\lambda): L(\mu)]$ for $r \gg 0$.
\end{prop}
\begin{proof}
By definition the space $\dual{M_r(\lambda)}$ is an $\lie h$-semisimple module
satisfying the LAC. Also it is isomorphic to $\dual{\SS^\bullet(\overline{\lie 
u}(r))} \ot \CC_\lambda$ as $\lie l[r]^+$-module, and hence lies in $\tilde 
\TT_{\lie l[r]}$ by Proposition \ref{p:ur-ladual}, and by Lemma \ref{prop:lcs} 
it has LCS. 
To see that it is $\lie n$-torsion, first observe that it is 
$\lie n(r) = \lie l[r] \cap \lie n$-torsion by virtue of being in $\tilde 
\TT_{\lie l[r]}$. Now recall the map $\theta$ from Example \ref{ex:gradings}. 
It induces a right-bounded grading on $A_r(\lambda)$, and since $\lie 
g^\theta_{>0} = \lie u(r)$, it follows that this subalgebra acts nilpotently on 
$A(\lambda)$and hence $A(\lambda)$ is $\lie n = \lie n(r) \niplus \lie 
u(r)$-torsion. This completes the proof for $A_r(\lambda)$.

Since $\CAT{}{}$ is closed by direct limits it follows that $A(\lambda)$
belongs to $\CAT{}{}$. Given a weight $\mu$ we build an LCS for $A(\lambda)$ 
at $\mu$ as follows: start with an LCS at $\mu$ for $A_r(\lambda)$ such that 
$A_r(\lambda)_\mu \to A(\lambda)_\mu$ is an isomorphism, and use the natural 
map to obtain a filtration of $A(\lambda)$. Adding $A(\lambda)$ at the top of 
the filtration gives us the desired LCS, and shows that the multiplicities 
coincide as desired.
\end{proof}

This result shows that $A(\lambda)$ is the projection of 
$M(\lambda)^\vee$ to $\CAT{}{}$. By Proposition \ref{prop:ss-dual-coind} we get 
that $A(\lambda) \cong \Phi(\ssCoind_{\overline{\lie b}}^{\lie g} 
\CC_\lambda) \cong \LAprojector(\ssCoind_{\overline{\lie b}}^{\lie g} 
\CC_\lambda)$. The following theorem shows that the $A(\lambda)$ have the 
usual properties associated to standard objects in highest weight categories.

\begin{thm}
\label{thm:standard}
For each $\lambda \in \lie h^\circ$ the following hold.
\begin{enumerate}[(i)]
\item
\label{i:pre-a1}
$A(\lambda)$ is indecomposable and $\soc A(\lambda) \cong L(\lambda)$.

\item 
\label{i:pre-a2}
The composition factors of $A(\lambda)/L(\lambda)$ are of the form $L(\mu)$ 
with $\mu \prec \lambda$.
\item 
\label{i:pre-a3}
For each $\mu \in \lie h^\circ$ we have $\dim \Hom_{\CAT{}{}} (A(\mu), 
A(\lambda)) < \infty$.
\end{enumerate}
\end{thm}
\begin{proof}
Let $\mu \in \lie h^\circ$. Since $\LAprojector$ is right adjoint to the 
inclusion of $\CAT{}{}$ in $\Mod{(\lie g, \lie h)}$ there are isomorphisms
\begin{align*}
\Hom_{\CAT{}{}}(L(\mu), A(\lambda)) 
	&\cong \Hom_{\lie g, \lie h}(L(\mu), \ssCoind_{\overline{\lie b}}^{\lie g} 
		\CC_\lambda) \cong \Hom_{\overline{\lie b}, \lie h}
			(L(\mu), \CC_\lambda),
\end{align*}
and since $L(\mu)$ has $\CC_\mu$ as its unique simple quotient as $\overline
{\lie b}$-module, it follows that this space has dimension $1$ when $\lambda = 
\mu$ and zero otherwise. Hence $L(\lambda)$ is the socle of $A(\lambda)$, which 
implies item (\ref{i:pre-a1}). Since the support of $A(\lambda)/L(\lambda)$ is
contained in the set of weights $\mu \prec \lambda$, item (\ref{i:pre-a2}) 
follows. 

Since $A(\mu)$ has a local composition series at $\lambda$, there exists a 
finite length $\lie g$-module $N \subset A(\mu)$ such that $(A(\mu)/N)_\lambda 
= 0$. Now consider the following exact sequence
\begin{align*}
	0 \to \Hom_{\CAT{}{}}(A(\mu)/N, A(\lambda)) \to 
		\Hom_{\CAT{}{}}(A(\mu), A(\lambda)) \to
		\Hom_{\CAT{}{}}(N, A(\lambda)).
\end{align*}
Since any nonzero map to $A(\lambda)$ must contain $L(\lambda)$ in its 
image, the first $\Hom$-space in the long exact sequence is zero and the map 
$\Hom_{\CAT{}{}}(A(\mu), A(\lambda)) \to \Hom_{\CAT{}{}}(N, A(\lambda))$ is 
injective. Now a simple induction shows that $\dim \Hom_{\CAT{}{}}(X, 
A(\lambda))$ is finite for any $X$ of finite length. This proves 
(\ref{i:pre-a3}).
\end{proof}
From now on we will refer to the $A(\lambda)$ as \emph{standard} modules of
$\CAT{}{}$. We point out that the order $\prec$ on $\lie h^\circ$ is not 
interval-finite, and so it is not yet clear that the $A(\lambda)$ are standard
in the sense of highest weight categories. In the coming subsections we will 
compute the Jordan-Holder multiplicities of the standard modules explicitly, 
and in the process find an interval-finite order for its simple constituents.

\subsection{The submodules $T_r(\lambda)$}
Our computation of the simple multiplicities of the module $A(\lambda)$ is 
rather long and technical. The idea is to find an exhaustion of $A(\lambda)$ 
which will allow us to compute the layers of the $\psi$-filtration of 
$A(\lambda)$. Set $T_r(\lambda) = \Phi_r(M_r(\lambda)^\vee)$; this is clearly a 
$\lie g_r$-module, and since it is spanned by $\lie h$-semisimple vectors it is 
in fact a $\lie g_r^+ = \lie g_r + \lie h$-module, and through $\psi$ we can 
see it as a graded module. 
\begin{lem}
For each $r \gg 0$ there are injective maps of $\psi$-graded $\lie 
g_r^+$-modules $T_r(\lambda) \hookrightarrow T_{r+1}(\lambda)$ and 
$T_r(\lambda) \hookrightarrow T(\lambda)$. Furthermore $A(\lambda) \cong 
\varinjlim T_r(\lambda)$ as $\lie g$-modules.
\end{lem}
\begin{proof}
Identifying $M_r(\lambda)^\vee$ with its image inside $M(\lambda)^\vee$ it is 
clear that 
$M_r(\lambda)^\vee \subset M_{r+1}(\lambda)^\vee$ and that 
\begin{align*}
T_r(\lambda) = 
\Phi_r(M_r(\lambda)^\vee) \subset \Phi_{r+1}(M_r(\lambda)^\vee) \subset 
\Phi_{r+1}(M_{r+1}(\lambda)^\vee) = T_{r+1}(\lambda).
\end{align*}
The desired maps are given by the inclusions, which trivially satisfy the 
statement. Now if $m \in \Phi_r(M(\lambda)^\vee)$ then there exists $s \geq r$ 
such that $m \in \Phi_r(M_s(\lambda)^\vee) \subset T_s(\lambda)$, so 
$A(\lambda) = \bigcup_{r \geq 0} T_r(\lambda) \subset M(\lambda)^\vee$.
\end{proof}

The definition of $\psi$-filtrations given in \ref{defn:psi-filtration} is
easily adapted to $\psi$-graded $\lie g_r^+$-modules. The fact that $A(\lambda)$
is the direct limit of the $T_r(\lambda)$ in the category of $\psi$-graded 
$\lie g_r^+$-modules implies that the $p$-th module in the $\psi$-filtration of 
$A(\lambda)$ is the limit of the $p$-th module in the $\psi$ filtrations of the
$T_r(\lambda)$. Furthermore, since direct limits are exact we can recover the 
$\lie s$-module $\mathcal S_p(A(\lambda))$ as the direct limit of the $\lie 
s_r$-modules $\mathcal S_p(T_r(\lambda))$. This will allow us to compute their 
characters and, using Proposition \ref{prop:s-filtration}, the simple 
multiplicities of $A(\lambda)$. 

\begin{rmk}
Notice that it is not true that the Jordan-Holder multiplicities of the 
$T_r(\lambda)$ can be recovered from the multiplicities of the $\lie 
s_r$-modules $\mathcal S_p(T_r(\lambda))$. Indeed, Proposition 
\ref{prop:s-filtration} follows from Proposition \ref{prop:hwm-properties}, and 
the analogous statement is clearly false for $\lie g_r$.
\end{rmk}
Set $R_{i} = \lie g^\psi_{\leq -i}$, i.e. the subspace of $\lie g$ spanned by 
all root-spaces corresponding to roots $\alpha$ with $\psi(\alpha) \leq -i$, 
and set $R_{i,r} = R_i \cap \lie g_r$. We see $R_i$ as $\lie b$-module through 
the isomorphism $R_i \cong \lie g/\lie g_{> -i}^\psi$. We set $\mathcal R = 
\bigoplus_{k=1}^n 2^{k-1} R_k$ and $\mathcal R_r = \bigoplus_{k=1}^n 2^{k-1} 
R_{k,r}$. 
\begin{lem}
\label{lem:W}
There exists an isomorphism of $\lie s_r$-modules
\begin{align*}
\mathcal S_p T_r(\lambda) 
	&\cong \left(\Ind^{\lie s_r}_{\lie s_r \cap \lie b} 
		\SS^\bullet(\mathcal R_r)_{p - \psi(\lambda)} 
		\ot \CC_\lambda \right)^\vee
\end{align*}

\end{lem}
The proof of this statement is quite technical and not particularly 
illuminating, so we postpone it until the end of this section. It is, however
the main step to compute the Jordan-Holder multiplicities of standard modules

\begin{thm}
\label{thm:standard-mults}
Let $\lambda$ and $\mu$ be eligible weights. Then 
\begin{align*}
[A(\lambda):L(\mu)] = m(\lambda + \nu, \mu) \dim \SS^\bullet(\mathcal 
	R^\vee)_\nu.
\end{align*} 
In particular this is nonzero if and only if $\mu$ is of the form $\sigma \cdot 
(\lambda + \nu)$ for a weight $\nu$ in the support of $\SS^\bullet(\mathcal 
R^\vee)$ and $\sigma \in \mathcal W^\circ$.
\end{thm}
\begin{proof}
Since $A(\lambda)^\psi_p = \bigcup_{r \geq 0} T_r(\lambda)^\psi_p$, it follows 
that $\mathcal F_p A(\lambda) = \varinjlim \mathcal F_p T_r(\lambda)$. Exactness
of direct limits of vector spaces implies that 
\begin{align*}
\frac{\mathcal F_p A(\lambda)}{\mathcal F_{p-1} A(\lambda)}
&\cong 
	\varinjlim \frac{\mathcal F_p T_r(\lambda)}{\mathcal F_{p-1} T_r(\lambda)}
\end{align*}
and so taking top degree components we get $\mathcal S_p A(\lambda) = \varinjlim
\mathcal S_p T_r(\lambda)$. It follows that $[A(\lambda): L(\mu)] = 
[T_r(\lambda): L_{\lie s_r}(\mu)]$ for large $r$. 

Using Lemma \ref{lem:W} and standard results on category $\mathcal O$ for the 
reductive algebra $\lie s_r$, see for example \cite{Humphreys08}*{3.6 Theorem}, $\mathcal S_p T_r(\lambda)$ has a filtration by dual Verma modules of the form 
$M_{\lie s_r}(\lambda + \nu)^\vee$, and each of these modules appears with 
multiplicity $\dim \SS^\bullet(\mathcal R^\vee)_\nu$. Thus 
\begin{align*}
[\mathcal S_p T_r(\lambda): L_{\lie s_r}(\mu|_r)]
	&= [M_{\lie s_r}(\lambda + \nu|_r): L_{\lie s_r}(\mu|_r)] 
		\dim \SS^\bullet(\mathcal R_r^\vee)_\nu, 
\end{align*}
and for large $r$ this is $m(\lambda + \nu, \mu) \dim \SS^\bullet(\mathcal 
R^\vee)_\nu$, as desired.
\end{proof}

\subsection{An interval finite order on $\lie h^\circ$}
We are now ready to show that there is an interval finite order for the simple 
constituents in standard objects of $\CAT{}{}$. In view of Theorem 
\ref{thm:standard-mults} there is only one reasonable choice.

\begin{defn}
Let $\lambda, \mu \in \lie h^\circ$. We write $\mu <_{\order} \lambda$
if there exist $\nu$ in the support of $W$ and $\sigma \in \mathcal W(\lie s)$ 
such that $\mu = \sigma \cdot (\lambda + \nu) \preceq \lambda$.
\end{defn}

\begin{lem}
\label{lem:order}
The order $<_{\order}$ is interval finite.
\end{lem}
\begin{proof}
Let us write $\mu \triangleleft \lambda$ if either 
\begin{enumerate}[(i)]
\item there is a simple root $\alpha$ such that $\mu = s_\alpha \cdot \lambda 
\prec \lambda$ or

\item $\mu = \lambda + \nu$ with $\nu$ a root such that $\psi(\nu) < 0$.
\end{enumerate}
Notice that $\mu <_{\order} \lambda$ implies $\mu \prec \lambda$. 
Since the support of $\mathcal R$ is the $\ZZ_{>0}$-span of the space of roots 
of $\lie g^{\psi}_{<0}$, it follows that $\triangleleft$ is a subrelation of 
$<_{\order}$, and in fact this order is the reflexive transitive closure of 
$\triangleleft$. To prove the statement, it is enough to show that there are 
only finitely many weights $\gamma$ such that $\mu \triangleleft \gamma 
<_\order \lambda$, and that there is a global bound on the length of chains of 
the form $\mu \triangleleft \lambda_1 \triangleleft \cdots \triangleleft 
\lambda_l \triangleleft \lambda$. 

If we are in case $(i)$ then $\gamma = s_\alpha \cdot \mu$, and since $\gamma 
\prec \lambda$ it follows that $\alpha \in \mathcal W(\lie s_r)$, so there are 
only finitely many options in this case. Notice also that $(\mu, \rho_{r+1})
< (\gamma, \rho_{r+1}) \leq (\lambda, \rho_{r+1})$, where $\rho_{r+1}$ is the
sum of all positive roots of $\lie s_{r+1}$. Thus in a chain as above there can
be at most $(\lambda - \mu, \rho_{r+1})$ inequalities of type $(i)$.

If we are in case $(ii)$, then $\lambda - \mu \succ 0$ and $\gamma =
\mu - \nu$ for some $\psi$-negative root $\nu$. Since $\lambda - \gamma =
\lambda - \mu - \nu \succeq 0$, it is enough to check that there are only 
finitely many $\nu$ such that the last inequality holds. Now since $\lambda$ and
$\mu$ are $r$-eligible, $(\lambda - \mu, \alpha) = 0$ for all roots outside of
the root space of $\lie g_{r+1}$. This means that $\nu$ must be a negative root 
of $\lie g_{r+1}$, of which there are finitely many. Notice also that in this 
case $\psi(\mu) < \psi(\gamma) \leq \psi(\lambda)$, and hence in any chain as 
above there are at most $\psi(\lambda - \mu)$ inequalities of type $(ii)$. Thus
any chain is of length at most $\psi(\lambda - \mu) + (\lambda - \mu, 
\rho_{r+1})$ and we are done.
\end{proof}

\section*{Appendix: A proof of Lemma \ref{lem:W}}

\subsection{A result on Schur functors}
We begin by setting some notation for Schur functors and Littlewood-Richardson 
coefficients. Recall that given two vector spaces $V, W$ and partitions 
$\lambda, \mu$ we have isomorphisms, natural in both variables,
\begin{align*}
\Schur_\lambda(V) \otimes \Schur_\mu(V)
	&= \bigoplus_{\nu} c_{\lambda, \mu}^\nu \Schur_{\nu}(V)
&\Schur_\lambda(V \oplus W) 
	&= \bigoplus_{\alpha, \beta} 
		c_{\alpha, \beta}^\gamma \Schur_\alpha(V) \ot \Schur_\beta(W).
\end{align*}

Given an $m$-tuple of partitions $\boldsymbol \alpha = (\alpha_1, \ldots, 
\alpha_m)$ and an $m$-tuple of vector spaces $U_1, \ldots, U_m$ we set 
$\Schur_{\boldsymbol \alpha}(U_1, \ldots, U_m) = \Schur_{\alpha_1}(U_1) 
\otimes \cdots \otimes \Schur_{\alpha_m}(U_m)$. It follows that for each
$\gamma \in \Part$ there exists $c_{\boldsymbol \alpha}^\gamma \in \ZZ_{\geq 
0}$ such that the following isomorphisms hold
\begin{align*}
\Schur_{\gamma}(U_1 \oplus \cdots \oplus U_m)
	&\cong \bigoplus_{\boldsymbol \alpha} c^\gamma_{\boldsymbol \alpha}
		\Schur_{\boldsymbol \alpha}(U_1, \ldots, U_m);\\
\Schur_{\boldsymbol \alpha}(U, U, \ldots, U)
	&\cong \bigoplus_{\gamma} c^\gamma_{\boldsymbol \alpha} 
		\Schur_{\gamma}(U). 
\end{align*}
These easily imply the following lemma.
\begin{lem}
\label{lem:schur-identity}
Fix $\boldsymbol \alpha \in \Part^m$. Given vector spaces $U_1, \ldots, U_m$ 
there is an isomorphism, natural in all variables,
\begin{align*}
\bigoplus_{\gamma \in \Part} \bigoplus_{\boldsymbol \beta \in \Part^m} 
	c_{\boldsymbol \alpha}^\gamma c_{\boldsymbol \beta}^\gamma
		&\Schur_{\boldsymbol \beta}(U_1, U_2, \ldots, U_m)
			\cong \Schur_{\boldsymbol \alpha}(U, U, \ldots, U)
\end{align*}
where $U = U_1 \oplus U_2 \oplus \cdots \oplus U_m$.
\end{lem}

\subsection{The module $\mathcal H$}
In order to understand the module $T_r(\lambda)$ we will need to look under
the hood of the dual modules $M_r(\lambda)^\vee$. For this subsection only we
write $\WW = \VV[r]$, so in particular $\WW^{(k)} = \VV^{(k)}[r]$ and 
$\WW^{(l)}_* = \VV^{(l)}_*[r]$, etc.

We begin by decomposing the parabolic algebra $\lie p(r)$ as $\lie a[r] \oplus 
\lie b_r \oplus \lie f_r$, where $\lie b_r = \lie b \cap \lie g_r,
\lie a[r] = \bigoplus_{k<l} \WW^{(k)} \otimes \WW^{(l)}$,
and $\lie f_r$ is the unique root subspace completing the decomposition. 
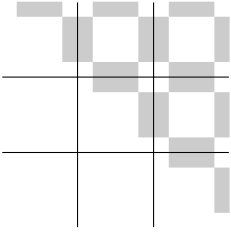
\begin{align*}
\begin{tikzpicture}
	\filldraw[draw=none,black!20!white]
		(0.2,3) rectangle (0.8,2.8)
		(1.2,3) rectangle (1.8,2.8)
		(2.2,3) rectangle (2.8,2.8)
		(1.2,2.2) rectangle (1.8,1.8)
		(2.2,2.2) rectangle (2.8,1.8)
		(2.2,1.2) rectangle (2.8,0.8)
		(0.8,2.8) rectangle (1.2,2.2)
		(1.8,2.8) rectangle (2.2,2.2)
		(2.8,2.8) rectangle (3,2.2)
		(1.8,1.8) rectangle (2.2,1.2)
		(2.8,1.8) rectangle (3,1.2)
		(2.8,0.8) rectangle (3,0.2)
		;
	\draw[step=1cm,black] (0.01,0.01) grid (2.99,2.99); 
	\node [below right, text width=6cm,align=justify] at
	(4,2.5) {The picture of $\lie f_r$. Notice that each rectangle is a direct summand of $\lie f_r$ as right $\lie l[r] \oplus \lie s_r$-module};
\end{tikzpicture}
\end{align*}

Using these spaces we can decompose $\lie g = \lie g[r] \oplus (\lie f_r \oplus 
\overline{\lie f}_r) \oplus \lie g_r$. This is a decomposition of $\lie g$ as 
$\lie g[r] \oplus \lie g_r$-module, since $\lie f_r \oplus \overline{\lie f}_r$ 
is stable by the adjoint action of this subalgebra and isomorphic to 
$(\WW^n \otimes (V_r)^*) \oplus (V_r \ot \WW_*^n)$. We also obtain a 
decomposition of $\overline {\lie b}_r \oplus \lie l[r] \oplus 
\overline{\lie a}[r]$-modules
\begin{align*}
\overline{\lie p}(r) &= \overline{\lie b}_r \oplus \overline{\lie f_r} 
	\oplus \lie l[r] \oplus \overline{\lie a}[r].
\end{align*}

Set $\lie q_r = \lie s_r + \lie b_r$. The subspaces $\lie a[r]$ and 
$\overline{\lie a}[r]$ are trivial $\lie q_r$-submodules of $\lie g$. The 
subspace $\lie f_r$ is also a $\lie q_r$-submodule of $\lie g$, and we endow 
$\overline{\lie f}_r$ with the structure of $\lie q_r$-modules through the 
vector space isomorphism $\overline{\lie f}_r = \lie f_r \oplus \overline{\lie 
f}_r / \lie f_r$. We set $\mathcal H = \SS^\bullet(\overline{\lie f}_r 
\oplus \overline{\lie a}[r])$. By the PBW theorem we have an isomorphism of 
$\lie g_r \oplus \lie l[r]$-modules
\begin{align*}
M_r(\lambda) &\cong U(\lie g) \ot_{\lie p(r)} \CC_\lambda
	\cong U(\lie g_r) \ot_{\lie b_r} (\mathcal H \ot \CC_\lambda)
	= \Ind_{\lie b_r}^{\lie g_r} \mathcal H \ot \CC_\lambda.
\end{align*}
where the $\lie l[r]$-module structure is determined by that of $\mathcal H$.
Thus when computing the $\lie l[r]'$-invariant vectors of the semisimple dual
of this module we get
\begin{align*}
T_r(\lambda) &= \Phi_r(M_r(\lambda)^\vee)
	\cong \ssCoind_{\overline{\lie b}_r}^{\lie g_r}
		\Phi_r(\mathcal H^\vee) \ot \CC_\lambda.
\end{align*}

\subsection{The invariants of $\mathcal H^\vee$}
Our task is thus to compute $\Phi_r(\mathcal H^\vee)$, and for this we will 
need a finer description of $\overline{\lie f}_r$. 
For each $k \in \interval{0,n}$ denote by $B_k = V_{r,+}^{(k)} \oplus 
V_{r,-}^{(k-1)}$ and $A_k = B_k^*$. These are $\lie s_r$-modules and we have a
decomposition of $\lie l[r] \oplus \lie s_r$-modules
\begin{align*}
\overline{\lie f}_r &= 
	\bigoplus_{k \geq l}  
		 (\WW^{(k)} \otimes A_l) \oplus (B_k \otimes\WW^{(l)}_*),\\
\SS^\bullet(\overline{\lie f}_r) &\cong
	\bigoplus_{\boldsymbol \alpha, \boldsymbol \beta}
		\bigotimes_{k \geq l}
			\Schur_{\alpha_{k,l}}(A_l) \otimes \Schur_{\beta_{k,l}}(B_k)
			\otimes \Schur_{\alpha_{k,l}}(\WW^{(k)}) \otimes 
				\Schur_{\beta_{k,l}}(\WW^{(k)}_*),
\end{align*}
where the sum runs over all families $\boldsymbol \alpha =(\alpha_{k,l})_{1 
\leq k \leq l \leq n}$ and $\boldsymbol \beta =(\beta_{k,l})_{1 \leq k \leq l 
\leq n}$. On the other hand we have decompositions
\begin{align*}
\overline{\lie a}[r] &= \bigoplus_{k > l} \WW^{(k)} \otimes \WW^{(l)}_*;&
\SS^\bullet(\overline{\lie a}[r]) &= 
	\bigoplus_{\boldsymbol \gamma} \bigotimes_{k > l} 
		\Schur_{\gamma_{k,l}}
			(\WW^{(k)}) \otimes \Schur_{\gamma_{k,l}}(\WW^{(l)}_*);
\end{align*}
where again the sum runs over all families of partitions $\boldsymbol \gamma =(
\gamma_{k,l})_{1 \leq l < k \leq n}$. We extend these families by setting 
$\gamma_{k,k} = \emptyset$, so $\gamma_{k,l}$ is defined for all $k \geq l$.

We now introduce the following notation: given a family of partitions 
$\boldsymbol \alpha =(\alpha_{k,l})_{1 \leq l \leq k \leq n}$ we set
$\boldsymbol \alpha_{k,\bullet} =(\alpha_{k,l})_{l \leq k \leq n}$ and 
$\boldsymbol \alpha_{\bullet,l} =(\alpha_{k,l})_{1 \leq l \leq k}$. By 
collecting all terms of the form $\Schur_{\eta_k}(\WW^{(k)}) \ot 
\Schur_{\nu_k}(\WW_*^{(k)})$ we obtain the following isomoprhism
of $\lie s_r \oplus \lie l[r]$-modules
\begin{align*}
\mathcal H &\cong
\bigoplus_{\boldsymbol{\alpha, \beta,\gamma,\eta,\nu}}
\bigotimes_{k=1}^n
c^{\eta_k}_{\boldsymbol \alpha_{k,\bullet}, \boldsymbol \gamma_{k,\bullet}} 
c^{\nu_k}_{\boldsymbol \beta_{\bullet,k}, \boldsymbol \gamma_{\bullet,k}}
Q(\boldsymbol \alpha_{k,\bullet}, \boldsymbol \beta_{\bullet,k}) \boxtimes
\Schur_{\eta_k}(\WW^{(k)}) \ot \Schur_{\nu_k}(\WW_*^{(k)})
\end{align*}
where the sum is taken over all collections of partitions and
\begin{align*}
Q(\boldsymbol \alpha_{k,\bullet}, \boldsymbol \beta_{\bullet,k})
	&= 
	\Schur_{\boldsymbol \alpha_{k,\bullet}}
		(A_1, A_2, \ldots, A_k) \ot \Schur_{\boldsymbol \beta_{\bullet,k}}
			(B_k, B_{k+1}, \ldots, B_n).
\end{align*}

Since this module is contained in $\SS^\bullet(\overline{\lie u}(r))$, it must
lie in $\tilde \TT_{\lie l[r]}$. It follows that each weight component 
intersects at most finitely many of these direct summands. Thus semisimple 
duality will commute with both the direct sum and the large tensor product. 
Furthermore, since each $Q(\boldsymbol \alpha_{k,\bullet}, \boldsymbol 
\beta_{\bullet,k})$ has finite dimensional components we have
\begin{align*}
\mathcal H^\vee 
	&\cong 
	\bigoplus_{\boldsymbol{\alpha, \beta,\gamma,\eta,\nu}}
\bigotimes_{k=1}^n
c^{\eta_k}_{\boldsymbol \alpha_{k,\bullet}, \boldsymbol \gamma_{k,\bullet}} 
c^{\nu_k}_{\boldsymbol \beta_{\bullet,k}, \boldsymbol \gamma_{\bullet,k}}
Q(\boldsymbol \alpha_{k,\bullet}, \boldsymbol \beta_{\bullet,k})^\vee \boxtimes
I(\boldsymbol \eta, \boldsymbol \nu)[r]^\vee.
\end{align*}
Now recall that $\Phi_r(I(\boldsymbol \eta, \boldsymbol \nu)[r]^\vee) \cong 
\ssHom_{\lie l[r]'}(I(\boldsymbol \eta, \boldsymbol \nu)[r], \CC) \cong
\delta_{\boldsymbol \eta, \boldsymbol \nu}$. Thus
\begin{align*}
\Phi_r(\mathcal H^\vee)
	&\cong 
	\bigoplus_{\boldsymbol{\alpha, \beta,\gamma,\eta}}
\bigotimes_{k=1}^n
c^{\eta_k}_{\boldsymbol \alpha_{k,\bullet}, \boldsymbol \gamma_{k,\bullet}} 
c^{\eta_k}_{\boldsymbol \beta_{\bullet,k}, \boldsymbol \gamma_{\bullet,k}}
Q(\boldsymbol \alpha_{k,\bullet}, \boldsymbol \beta_{\bullet,k})^\vee \\
	&\cong \left( \bigoplus_{\boldsymbol{\alpha, \beta,\gamma,\eta}}
\bigotimes_{k=1}^n
c^{\eta_k}_{\boldsymbol \alpha_{k,\bullet}, \boldsymbol \gamma_{k,\bullet}} 
c^{\eta_k}_{\boldsymbol \beta_{\bullet,k}, \boldsymbol \gamma_{\bullet,k}}
Q(\boldsymbol \alpha_{k,\bullet}, \boldsymbol \beta_{\bullet,k})\right)^\vee.
\end{align*}

Denote the module inside the semisimple dual by $Q$, and let us analyse each 
factor of the tensor product separately. When $k = 1$ the 
Littlewood-Richardson coefficient $c^{\eta_k}_{\boldsymbol 
\alpha_{k,\bullet}, \boldsymbol \gamma_{k,\bullet}}$ simplifies to 
$c^{\eta_1}_{\alpha_{1,1}, \emptyset}$, and the only way for this to be nonzero 
is that $\eta_1 = \alpha_{1,1}$. Thus the first factor in the tensor product has
the form
\begin{align*}
c^{\eta_1}_{\boldsymbol \beta_{\bullet,1}, \boldsymbol \gamma_{\bullet,1}}
	\Schur_{\eta_1}(A_1) \ot \Schur_{\boldsymbol \beta_{\bullet,1}}
		(B_1, B_{2}, \ldots, B_n) .
\end{align*}
We can sum all these factors over $\eta_1$ and $\beta_{l,1}$ for every $l \geq 
1$ since these do not appear in any other factors. Using Lemma 
\ref{lem:schur-identity} we obtain
\begin{align*}
\bigoplus_{\boldsymbol \beta_{\bullet,1}, \eta_1}
&c^{\eta_1}_{\boldsymbol \beta_{\bullet,1}, \boldsymbol \gamma_{\bullet,1}}
\Schur_{\eta_1}(A_1) \ot \Schur_{\boldsymbol \beta_{\bullet,1}}
		(B_1, B_{2}, \ldots, B_n) \\
	&\cong \bigoplus_{\boldsymbol \beta_{\bullet,1}}
	 	\Schur_{\boldsymbol \beta_{\bullet,1}}
		(A_1, A_{1}, \ldots, A_1)
	\otimes \Schur_{\boldsymbol \gamma_{\bullet,1}}
		(A_1, A_{1}, \ldots, A_1)\otimes
	\Schur_{\boldsymbol \beta_{\bullet,1}}
		(B_1, B_{2}, \ldots, B_n)  \\
	&\cong \SS^\bullet(A_1 \ot (B_1 \oplus B_2 \oplus \cdots \oplus B_n))
		\otimes \Schur_{\boldsymbol \gamma_{\bullet,1}}
		(A_1, A_{1}, \ldots, A_1).
\end{align*}
We see thus that $Q$ is isomorphic to
\begin{align*}
\SS^\bullet(A_1 \otimes (B_1 \oplus B_2 \oplus \cdots \oplus B_n))
\otimes \bigoplus_{\boldsymbol{\alpha, \beta,\gamma,\eta}}
\bigotimes_{k=2}^n
c^{\eta_k}_{\boldsymbol \alpha_{k,\bullet}, \boldsymbol \gamma_{k,\bullet}} 
c^{\eta_k}_{\boldsymbol \beta_{\bullet,k}, \boldsymbol \gamma_{\bullet,k}}
\Schur_{\gamma_{k,1}}(A_1)\otimes Q(\boldsymbol \alpha_{k,\bullet}, 
\boldsymbol \beta_{\bullet,k})
\end{align*}
where the sum is now restricted to the subsequences where $k \geq 2$. 

Let us now focus on the new factor corresponding to $k = 2$. We can again sum 
over $\eta_2$ and $\boldsymbol \beta_{\bullet,2}$ to get
\begin{align*}
\bigoplus_{\eta_2, \boldsymbol \beta_{\bullet,2}}
&c^{\eta_2}_{\boldsymbol \beta_{\bullet,2}, \boldsymbol \gamma_{\bullet,2}} 
c^{\eta_2}_{\alpha_{2,1}, \alpha_{2,2}, \gamma_{2,1}}
	\Schur_{\gamma_{k,1}}(A_1) \otimes \Schur_{\alpha_{2,1}, 
	\alpha_{2,2}}(A_1, A_2)\otimes 
		\Schur_{\boldsymbol \beta_{\bullet,2}}(B_2, \ldots, B_n)\\&\cong
\bigoplus_{\boldsymbol \beta_{\bullet,2}}
	 \Schur_{\boldsymbol \beta_{\bullet,2}}(\overline A_2)
		\otimes \Schur_{\boldsymbol \gamma_{\bullet,2}}(\overline A_2)
		\otimes	\Schur_{\boldsymbol \beta_{\bullet,2}}(B_2, \ldots, B_n)
		\\
&\cong \SS^\bullet(\overline A_2 \otimes (B_2 \oplus \cdots \oplus B_n))
\otimes \Schur_{\boldsymbol \gamma_{\bullet,2}}(\overline A_2).
\end{align*}
where $\overline A_2 = A_1 \oplus A_1 \oplus A_2$. After $n$ steps of this 
recursive process we obtain 
\begin{align*}
Q \cong 
\bigotimes_{k=1}^n \SS^\bullet(\overline A_k \otimes \underline B_k)
\cong \SS^\bullet \left(\bigoplus_{k=1}^n \overline A_k \otimes \underline B_k  
\right)
\end{align*}
where
\begin{align*}
\underline B_k &= B_k \oplus B_{k+1} \oplus \cdots \oplus B_n
\end{align*}
and
\begin{align*}
\overline A_k &= A_1 \oplus A_2 \oplus \cdots \oplus A_k \oplus
	\overline A_1 \oplus \overline A_2 \oplus \cdots \oplus \overline A_{k-1}\\
	&= 2^{k-1}A_1 \oplus 2^{k-2}A_2 \oplus \cdots \oplus A_k.
\end{align*}
Thus in the direct sum the term $A_k \otimes B_l$ appears $2^{l-k} + 2^{l-k-1}
+ \cdots + 2 + 1 = 2^{l-k+1}-1$ times. Now the sum of all the $A_k \otimes B_l$
with $l - k + 1= i$ fixed is the space $R_{i,r}$ introduced in the preamble of 
Lemma \ref{lem:W} and so $Q \cong \mathcal \SS^\bullet(\mathcal R_r)$. The 
naturality of Schur functors implies that this is an isomorphism of $\lie 
s_r$-modules. 

Now recall that $T_r(\lambda) = \ssCoind_{\overline{\lie b_r}}^{\lie g_r} 
Q^\vee \ot \CC_\lambda$. The $p$-th module of the $\psi$-filtration of 
$T_r(\lambda)$ is thus given by $\ssCoind_{\overline{\lie b_r}}^{\lie g_r} 
Q_{\geq p-\psi(\lambda)}^\vee \ot \CC_\lambda$, and by exactness of 
semisimple coiniduction on duals, the $p$-th layer of the filtration is 
isomorphic to $\ssCoind_{\overline{\lie b_r}}^{\lie g_r} Q_{p - 
\psi(\lambda)}^\vee \ot \CC_\lambda$, whose top layer is in turn isomorphic to 
$\ssCoind_{\overline{\lie b} \cap \lie s_r}^{\lie s_r} Q_{p-\psi(\lambda)}^\vee 
\ot \CC_\lambda$. Finally, the isomorphism we found before tells us that 
$Q_{\geq p-\psi(\lambda)}^\vee \ot \CC_\lambda \cong \SS^\bullet(\mathcal 
R^\vee)_{p-\psi(\lambda)}$ as $\lie s_r \cap \overline{\lie b}$-modules, so we 
are done.

\section{Category $\CAT{}{}$: injective objects and further categorical 
properties}
\label{s:injectives}

\subsection{Large annihilator coinduction}
In previous sections we have used inflation functors to turn objects in 
$\overline{\mathcal O}_{\lie s}$ into $\lie q$-modules and then used induction 
to construct objects in $\CAT{}{}$. In this section we will use a dual 
construction. Recall that $I_{\lie s}(\lambda)$ denotes the injective envelope
of $L_{\lie s}(\lambda)$ in $\overline{\mathcal O}_{\lie s}$.
\begin{defn}
We denote by $\mathcal C: \overline{\mathcal O}_{\lie s} \to \Mod{(\lie g, 
\lie h)}$ the functor $\mathcal C = \Phi \circ \ssCoind_{\overline{\lie 
q}}^{\lie g} \circ \mathcal I_{\lie s}^{\overline{\lie q}}$. For each 
$\lambda \in \lie h^\circ$ we set $I(\lambda) = \mathcal C(I_{\lie 
s}(\lambda))$.
\end{defn}
Notice that $\mathcal C$ does not automatically fall in $\CAT{}{}$. However 
from the definition and Proposition \ref{prop:ss-dual-coind} it follows
that $\mathcal C(M(\lambda)^\vee) = A(\lambda)$, so standard objects from the
category $\overline{\mathcal O}_{\lie s}$ are sent to standard objects of 
$\CAT{}{}$. We will show that the same holds for injective modules, and that
in fact the standard filtrations of injective modules in $\overline{\mathcal 
O}_{\lie s}$ also lift to standard filtrations in $\CAT{}{}$.

\subsection{Injective objects}
We now begin our study of injective envelopes in $\CAT{}{}$. Set $J(\lambda) = 
\ssCoind_{\overline{\lie q}}^{\lie g} \mathcal I_{\lie s}^{\overline{\lie q}} 
I_{\lie s}(\lambda)$. We will show that, just as for standard modules, we have
$I(\lambda) = \Phi(J(\lambda)) = \LAprojector(J(\lambda))$, and that this is
the injective envelope of the simple module $L(\lambda)$.

\begin{lem}
\label{lem:coind-inj}
Let $\lambda \in \lie h^\circ$ and let $I_{\lie s}(\lambda)$ be the injective 
envelope of $L_{\lie s}(\lambda)$ in $\overline{\mathcal O}_{\lie s}$. Then 
for any object $M$ in $\CAT{}{}$ we have $\Ext_{\lie g, \lie h}^1 (M, J(\lambda)
) = 0$.
\end{lem}
\begin{proof}
It is enough to prove the result for $M = L(\mu)$ with $\mu$ an eligible 
weight. It follows from Proposition \ref{prop:hwm-properties} that $L(\mu) \cong
\Ind_{\lie s}^{\overline{\lie q}} L_{\lie s}(\lambda)$ as $\overline{\lie 
q}$-module. Since $I_{\lie s}(\lambda)$ is an injective object it is acyclic 
for the functor $\ssCoind_{\overline{\lie q}^{\lie g}} \circ \mathcal I_{\lie 
s}^{\overline{\lie q}}$ so by standard homological algebra
\begin{align*}
\Ext_{\lie g, \lie h}^1(L(\mu), J(\lambda))
	&\cong \Ext_{\lie s, \lie h}^1(U(\lie s) \ot_{\overline{\lie q}} L(\mu),
		I_{\lie s}(\lambda))
	= \Ext_{\lie s, \lie h}^1(L_{\lie s}(\lambda), I_{\lie s}(\lambda)).
\end{align*}
Since $\overline{\mathcal O}_{\lie s}$ is closed under semisimple extensions,
this last $\Ext$-space must be $0$.
\end{proof}

The class of $\Phi$-acyclic modules is closed by extensions, and so is the 
subclass of $\Phi$-acyclic modules $M$ such that $\Phi(M) = \LAprojector(M)$. 
The following lemma states that any module with a filtration by dual Verma
modules lies in this class.

\begin{lem}
\label{lem:phi-pi}
Let $M$ be a weight $\lie g$-module. Suppose $M$ has a finite filtration 
$F_1M \subset F_2M \subset \cdots F_nM$ such that its $i$-th layer is 
isomorphic to $M(\lambda_i)^\vee$ for $\lambda_i \in \lie h^\circ$. Then 
$\Phi(M) = \LAprojector(M)$, and $\LAprojector(F_i M)/\LAprojector(F_{i-1}M)
\cong A(\lambda_i)$.
\end{lem}
\begin{proof}
The proof reduces to showing that $M(\lambda)^\vee$ is acyclic for both
$\Phi$ and $\Gamma_{\lie n}$. 

Recall that $\Phi$ can be written as the direct limit of functors $\Phi_r$
taking $\lie l[r]'$-invariant submodules, and that there are natural 
isomorphisms
\begin{align*}
\Phi_r &\cong 
	\bigoplus_\mu \Hom_{\lie g, \lie h}(U(\lie g) \ot_{\lie l[r]^+} 
		\CC_\mu, -)
\end{align*}
where the sum is taken over all $r$-eligible weights $\mu \in \lie h^\circ$. 
Since derived functors commute with direct limits, and in particular with 
direct sums, it is enough to show that if $\lambda$ and $\mu$ are $r$-eligible
then $\Ext^1_{\lie g, \lie h}(U(\lie g) \ot_{\lie l[r]^+} \CC_\mu, 
M(\lambda)^\vee) = 0$. 

Notice that the result is trivial if $\lambda$ and $\mu$ have different levels,
so we may assume that both have level $\chi = \chi_1 \omega^{(1)} + \cdots + 
\chi_n \omega^{(n)}$ and so $\lambda = \lambda_r + \chi$ and $\mu = \mu|_r + 
\chi$. We can rewrite $M(\lambda)^\vee$ as follows
\begin{align*}
M(\lambda)^\vee
 	&\cong \ssCoind_{\overline{\lie p}(r)}^{\lie g}
 		\mathcal I^{\overline{\lie p}(r)}_{\lie l[r]^+} 
 		\ssCoind_{\overline{\lie b} \cap \lie l[r]^+}^{\lie l[r]^+}
 			(\CC_{\chi_1} \boxtimes \cdots \boxtimes \CC_{\chi_n}) \ot 
 				\CC_{\lambda_r}\\
 	&\cong \ssCoind_{\overline{\lie p}(r)}^{\lie g}
 		\mathcal I^{\overline{\lie p}(r)}_{\lie l[r]^+} (M_1^\vee
 		 			\boxtimes \cdots \boxtimes M_n^\vee) \otimes 
 		 				\CC_{\lambda_r}
\end{align*}
where $M_i = M_{\lie l[r]}(\chi_i \omega^{(i)}[r])$.
Using that inflation functors are exact and semisimple coinduction is 
exact on semisimple duals, standard homological algebra implies
\begin{align*}
\Ext^1_{\lie g, \lie h}&(U(\lie g) \ot_{\lie l[r]^+} \CC_\mu, M(\lambda)^\vee) \\
	&\cong \Ext_{\lie l[r]^+, \lie h}^1(U(\lie l[r])^+ 
		\ot_{\overline{\lie p}(r)}
		U(\lie g) \ot_{\lie l[r]^+} \CC_\mu, 
		M_1^\vee \boxtimes \cdots \boxtimes M_n^\vee \otimes 
 		 				\CC_{\lambda_r}) \\
 	&\cong \Ext_{\lie l[r], \lie h}^1
 		(\SS^\bullet(\lie u(r)) \ot \CC_{\mu - \lambda}, 
 			M_1^\vee \boxtimes \cdots \boxtimes M_n^\vee).
\end{align*}
The left hand side of this $\Ext$-space decomposes as a direct sum of tensor
$\lie l[r]$-modules, so the $\Ext$-space decomposes as a tensor product
of spaces of the form 
\begin{align*}
\Ext_{\lie g(\VV^{(k)}[r]), \lie h^{(k)}}(T_k, M_k^\vee)
\end{align*}
with $T_k$ a tensor module, and hence in $\CAT{\lie l[r]}{\lie l[r]}$. Now 
each $M_i^\vee$ is injective in the corresponding category $\overline{\mathcal 
O}$, since it is the dual Verma module corresponding to a weight that is 
maximal in its dot-orbit. By Lemma \ref{lem:coind-inj} each of these 
$\Ext$-spaces is $0$, and hence $M(\lambda)^\vee$ is $\Phi$-acyclic. 

Now take $J_r$ to be the left ideal of $U(\lie g)$ generated by $\lie n^r$. 
There exists a natural isomorphism
\begin{align*}
\Gamma_{\lie n} 
	\cong \varinjlim \bigoplus_{\lambda \in \lie h^\circ}
		\Hom_{\lie g, \lie h}(U(\lie g)/J_r \ot_{\lie h} \CC_\lambda, -),
\end{align*}
so a similar argument works in this case. Finally since $\LAprojector = 
\Gamma_{\lie n} \circ \Phi$, we are done.
\end{proof}

We are now ready to prove the main result of this section.
\begin{thm}
\label{thm:injectives}
Let $\lambda \in \lie h^\circ$. The module $I(\lambda)$ lies in $\CAT{}{}$ and
it is an injective envelope for $L(\lambda)$. Furthermore it has a finite 
filtration whose first layer is $A(\lambda)$ and whose higher layers are 
isomorphic to $A(\mu)$ with $\mu >_\fin \lambda$.
\end{thm}
\begin{proof}
Since inflation is exact and semisimple coinduction is exact over semisimple
duals, the costandard filtration of $I_{\lie s}(\lambda)$ induces a finite 
filtration over $J(\lambda)$, whose first layer is isomorphic to 
$M(\lambda)^\vee$ and whose higher layers are isomorphic to $M(\mu)^\vee$ for 
$\mu \succeq \lambda$.
Applying Lemma \ref{lem:phi-pi} it follows that $I(\lambda) = \Phi(J(\lambda))
= \LAprojector(J(\lambda))$ and hence belongs to $\CAT{}{}$, and that this
module has the required filtration. Finally, reasoning as in the proof of Lemma 
\ref{lem:coind-inj} we have isomorphisms
\begin{align*}
\Ext_{\CAT{}{}}^i(L(\mu), I(\lambda)) \cong
\Ext_{\lie g, \lie h}^i(L(\mu), J(\lambda))
	&\cong \Ext_{\lie s, \lie h}^i(L_{\lie s}(\mu), I_{\lie s}(\lambda)).
\end{align*}
Taking $i = 0$ and varying $\mu$ we see that $L(\lambda)$ is the socle of 
$I(\lambda)$, and taking $i = 1$ it follows that $I(\lambda)$ is injective in 
$\CAT{}{}$.
\end{proof}

\subsection{Highest weight structure and blocks}
We are now ready to prove the main result of this paper, namely that $\CAT{}{}$
is a highest weight category in the sense of Cline, Parshall and Scott 
\cite{CPS88}.
\begin{thm}
Category $\CAT{}{}$ is a highest weight category with indexing set 
$(\lie h^\circ, <_\order)$. Simple modules are simple highest weight modules 
$L(\lambda)$, the family of standard objects $A(\lambda)$ have infinite length,
and the corresponding injective envelopes $I(\lambda)$ have finite standard 
filtrations.
\end{thm}
\begin{proof}
We only need to refer to previously proved results. We have just seen in Lemma 
\ref{lem:order} that $<_\order$ is an interval-finite order. The fact that 
$\CAT{}{}$ is locally artinian is immediate from \ref{prop:fg-fl}. That 
eligible weights parametrise simple modules was proved in Theorem 
\ref{thm:simples}, and standard modules exist and have the desired properties 
by Theorem \ref{thm:standard}. Finally injective modules exist and have finite 
standard filtrations with the required properties by Theorem 
\ref{thm:injectives}.
\end{proof}

A consequence of the previous results is that $M_{\lie s}(\lambda)^\vee$ is 
$\mathcal C$-acyclic, and standard filtrations of $I(\lambda)$ arise as the 
image by $\mathcal C$ of a standard filtration of $I_{\lie s}(\lambda)$. As a 
consequence we get the following analogue of BGG-reciprocity in $\CAT{}{}$.
\begin{cor}
Let $\lambda, \mu \in \lie h^\circ$. The multiplicity of $A(\mu)$ in any 
standard filtration of $I(\lambda)$ equals $m(\lambda, \mu)$.
\end{cor}
\begin{proof}
It is enough to check the result for one standard filtration.
We can obtain a standard filtration of $I(\lambda)$ by applying the functor 
$\mathcal C$ to a standard filtration of $I_{\lie s}(\lambda)$. Thus the 
desired multiplicity is the same as the multiplicity of $M_{\lie s}(\mu)^\vee$ 
in $I_{\lie s}(\lambda)$. By Theorem \ref{thm:inj-filtration} this multiplicity 
equals $[M_{\lie s}(\mu): L_{\lie s}(\lambda)] = m(\lambda,\mu)$. 
\end{proof}

We finish with a description of the irreducible blocks of $\CAT{}{}$. Recall 
that we denote by $\Lambda$ the root lattice of $\lie g$. For each $\kappa
\in \lie h^\circ / \Lambda$ set $\CAT{}{}(\kappa)$ to be the subcategory formed 
by those modules in $\CAT{}{}$ whose support is contained in $\kappa$. Notice 
in particular that each class has a fixed level, and so the decomposition 
$\CAT{}{} = \prod_{\kappa} \CAT{}{}(\kappa)$ refines the decomposition by 
levels. We now show that these are the indecomposable blocks of $\CAT{}{}$.
\begin{prop}
The blocks $\CAT{}{}(\kappa)$ are indecomposable.
\end{prop}
\begin{proof}
The support of $\lie g_{-1}^\psi \subset \mathcal R$ spans the 
root lattice, and from Theorem \ref{thm:standard-mults} that $[A(\lambda): 
L(\lambda)] = [A(\lambda): L(\lambda + \mu)]$ for any root in $\lie 
g^\psi_{-1}$. Thus whenever two weights $\lambda, \mu$ lie in the same class
modulo the root lattice, there is a finite chain of weights $\lambda = 
\lambda_1, \lambda_2, \ldots, \lambda_m = \mu$ such that $L(\lambda_{i+1})$ is a
simple constituent of the indecomposable module $A(\lambda_i)$. This shows that 
all simple modules in the block $\CAT{}{}(\kappa)$ are linked and hence the 
block is indecomposable.
\end{proof}

\end{document}